\documentclass{amsart}

\usepackage{amsmath,amssymb,amsthm}
\usepackage{enumerate}
\usepackage{cite}
\usepackage{pgfplots}
\usepackage{ifpdf}
\usepackage[foot]{amsaddr}

\usepackage{hyperref}
\hypersetup{
  linkbordercolor=[rgb]{0,.3,1},
  colorlinks=true,
  linkcolor=[rgb]{0,.3,.5},
  citecolor=[rgb]{0,.3,1}
}

\newcommand{\TheTitle}{%
  A framework for structured linearizations of matrix polynomials
  in various bases
} 

\newcommand{\TheAuthors}{L. Robol, R. Vandebril, P. Van Dooren}

\ifpdf
\hypersetup{
  pdftitle={\TheTitle},
  pdfauthor={\TheAuthors}
}
\fi

\renewcommand{\leq}{\leqslant}
\renewcommand{\geq}{\geqslant}

\theoremstyle{theorem}
\newtheorem{theorem}{Theorem}
\newtheorem{lemma}[theorem]{Lemma}
\newtheorem{corollary}[theorem]{Corollary}

\theoremstyle{definition}
\newtheorem{definition}[theorem]{Definition}

\theoremstyle{remark}  
\newtheorem{remark}[theorem]{Remark}
\newtheorem{numerical-experiment}{Numerical experiment}

\newcommand{\x}{\lambda}

\DeclareMathOperator{\rev}{rev}

\newcommand{\pf}{\otimes}

\newcommand{\prm}{\underline \times}

\DeclareMathOperator{\diag}{diag}

\DeclareMathOperator{\bsd}{bds}

\newcommand{\norm}[1]{\lVert #1 \rVert}

\begin{document}
  
  \begin{abstract}
     We present a framework for the construction of linearizations
    for scalar and
    matrix polynomials based on dual bases which, in the 
    case of orthogonal polynomials, can be described 
    by the associated recurrence relations. 
    The framework provides an extension of the classical
    linearization theory for polynomials expressed in non-monomial bases and
    allows to represent polynomials expressed in product families, that is 
    as a linear combination of elements of the form $\phi_i(\x) \psi_j(\x)$, where
    $\{ \phi_i(\x) \}$ and $\{ \psi_j(\x) \}$ can either be polynomial bases or 
    polynomial families which satisfy some mild assumptions. 
    
    We show that this general construction can be used for many different 
    purposes. Among them, we show how to linearize sums of polynomials
    and rational functions expressed in different bases. As an example, 
    this allows to
    look for intersections of functions interpolated on different nodes
    without converting them to the same basis. 
    
    We then provide some constructions for structured linearizations for
    $\star$-even and $\star$-palindromic matrix polynomials. The extensions
    of these constructions to
     $\star$-odd and $\star$-antipalindromic of odd degree
    is discussed and follows immediately from the previous results. 

  \end{abstract}  
    
  \author{Leonardo Robol}
  \address[Leonardo Robol]{Department of Computer Science, KU Leuven,  Celestijnenlaan 200A, 3001 Heverlee, Belgium.}
  \email{leonardo.robol@cs.kuleuven.be}
  
  \author{Raf Vandebril}
  \address[Leonardo Robol]{Department of Computer Science, KU Leuven,  Celestijnenlaan 200A, 3001 Heverlee, Belgium.}  
  \email{raf.vandebril@cs.kuleuven.be}
  
  \author{Paul Van Dooren}
  \address[Paul Van Dooren]{ICTEAM, Universite Catholique de Louvain, Batiment Euler, Avenue G. Lemaitre 4, B-1348 Louvain-la Neuve, Belgium.}
  \email{paul.vandooren@uclouvain.be}     
  
  \title[Linearizations of matrix polynomials in various bases]{
    A framework for structured linearizations of matrix polynomials
    in various bases
  }
    
      
  \maketitle 
  
\section{Introduction}
  
  In recent years much interest has been devoted to finding linearizations
  for polynomials and matrix polynomials. The Frobenius linearization, i.e.,
  the classical companion, 
  has been the de-facto standard in polynomial eigenvalue problems and 
  polynomial rootfinding for a long time \cite{gantmacher,gohberg1982matrix}.
  Nevertheless, 
  recently much work has been put into developing other families of linearizations. 
  Among these some linearizations preserve spectral
  symmetries available in the original problem \cite{higham2006symmetric,mackey2006structured,mackey2009numerical}, others
  linearize matrix polynomials formulated in non-monomial bases
  \cite{amiraslani2009linearization, corless2007generalized} and also
  some variations are based on an idea of Fiedler about decomposing companion
  matrices into products of simple factors \cite{antoniou2004new,de2010fiedler,de2012fiedler,fiedler2003note}.  
  
  In this work we take as inspiration the results of 
  Dopico, P\'erez, Lawrence and Van Dooren \cite{ldvdp}
  that characterize the structure of some permuted Fiedler linearizations
  by using dual minimal bases \cite{forney1975minimal}. 
  We extend the results
  in a way that allows us to deal with many more formulations, 
  and we use it to derive numerous different linearizations. 
  We also use these examples
  to prove the effectiveness of this result as a tool for constructing
  structured linearizations (thus preserving spectral symmetries
  in the spirit of the works cited above) and also linearizations
  for sums of polynomials and rational functions. 
  
  In particular, we consider the rootfinding case of polynomials 
  that are expressed as linear combination of elements in a so-called
  product family; the most common case where this can be applied is when
  considering two different polynomial bases $\{ \phi_i \}$
  and $\{ \psi_j \}$ and representing polynomials as sums of objects
  of the form $\phi_i(\x) \psi_j(\x)$. This apparently artificial construction
  has, however, many interesting applications, such as finding intersections
  of polynomials and rational functions defined in different bases. 
  
  In Section~\ref{sec:lin-productbases} we give a formal definition
  of what we call
  a \emph{product family of polynomials}, denoted by 
  $\phi \pf \psi$. We define the vector
  $\pi_{k, \phi}(\x)$ to be the one with the elements of the family
  as entries and we show that $\pi_{k, \phi \pf \psi}(\x)$ is 
  given by $\pi_{k, \phi}(\x) \otimes 
  \pi_{k, \psi}(\x)$.   
  We present a theorem that allows to linearize every polynomial written
  as a linear combination of elements in a product family, and we also
  generalize the construction to the product of more than
  two families in Section~\ref{sec:morebases}.   
  We consider a certain class of \emph{dual polynomial bases}
  (with the notation of the classical work by Forney \cite{forney1975minimal})
  of a polynomial vector $\pi_{k, \phi}(\x)$, 
  which we identify with linear matrix polynomials $L_{k, \phi}(\x)$
  such that $L_{k, \phi}(\x) \pi_{k, \phi}(\x) = 0$, which will be used as
  a tool to build linearizations.   
  At the end of the section we introduce an explicit construction for
  linearizing polynomial families arising from orthogonal
  and interpolation bases. We cover the case of every polynomial
  basis endowed with a recurrence relation, and we provide explicit 
  constructions for the Lagrange, Newton, Hermite and 
  Bernstein cases. We describe
  the dual bases for all these cases and, as shown
  by Theorem~\ref{thm:twobases}, they are the only ingredient required
  to build the linearizations. 
  
  The rest of the paper deals with the problem of exploiting this freedom
  of choice to obtain many interesting
  results. 
  
  In Section~\ref{sec:sums} we show how
  to linearize the sum of two scalar polynomials or rational functions expressed
  in different bases, without the need of an explicit basis conversion. 
  This can have important applications in the cases where interpolation
  polynomials are obtained from experimental data (that cannot be resampled
  - so there is no choice for the interpolation basis) 
  or in cases where an explicit change of basis is badly conditioned. 
  
  Infinite eigenvalues may appear when linearizing
  the sum of polynomials. We report numerical experiments
  that show that they do not affect the numerical 
  robustness of the approach in many cases. 
  Moreover, we show that for the rational
  case, under mild hypotheses, it is possible to construct strong linearizations
  which do not have spurious infinite eigenvalues. 
  
  
  In Section~\ref{sec:structure-preservation} we turn our attention to 
  preserving spectral symmetries and we provide explicit constructions
  for linearizations of $\star$-even/odd and $\star$-palindromic matrix polynomials. 
  We show that a careful choice of the dual bases
  for use in Theorem~\ref{thm:twobases} yields linearizations with
  the same spectral symmetries of the original matrix polynomial. 
  
  Finally, in Section~\ref{sec:infinity-deflation}, we describe
  a numerical approach to deflate the infinite eigenvalues that are 
  present in some of the constructions, based on the staircase
  algorithm of Van Dooren \cite{van1979computation}. 
  
  In Section~\ref{sec:conclusions} we draw some conclusions
  and we propose some possible development for future research. 
  
  \section{A general framework to build linearizations}
    \subsection{Notation}
    
    In the following we will often work with the vector space
    of polynomials of degree at most $k$ on a field $\mathbb F$,
    denoted as $\mathbb F_k[\x]$. 
        
    In the study of strong linearizations is also important to consider
    the $\rev$ operator, which reverses the coefficients of the
    polynomial when represented in the monomial basis. 
    
    \begin{definition}
      Given a non zero 
      matrix polynomial $P(\x) = \sum_{i = 0}^n P_i \x^i$ we define
      its \emph{degree} as the largest integer $i \geq 0$ such that
      $P_i \neq 0$, that is the maximum of all the degrees of the entries
      of $P(\x)$. We denote it 
      by $\deg P(\x)$. 
    \end{definition}
    
    \begin{definition}
      Given a matrix polynomial $P(\x)$, its \emph{reversed polynomial}, 
      denoted by $\rev P(\x)$, is defined by 
      $\rev P(\x) := x^{\deg P(\x)} P(\x^{-1})$. 
    \end{definition}
    
    Intuitively, a linearization for a matrix polynomial $P(\x)$ 
    is a linear matrix polynomial $L(\x)$ such that $L(\x)$ is 
    singular only when $P(\x)$ is. However, this is not sufficient
    in most cases since there is also the need to match eigenvectors
    and partial multiplicities, so the definition has to be 
    a little more involved. We refer to the work of 
    De Ter\'an, Dopico and Mackey \cite{de2014spectral} for a 
    complete overview of the subject.
    
    \begin{definition}
      A matrix polynomial $E(\x)$ is said to be \emph{unimodular} if 
      it is invertible in the ring of matrix polynomials or, equivalently, 
      if $\det E(\x)$ is a non zero constant of the field. 
    \end{definition}
    
    \begin{definition}[Extended unimodular equivalence]
      Let $P(\x)$ and $Q(\x)$ be matrix polynomials. 
      We say that they 
      are \emph{extended unimodularly equivalent}
      if there exist positive integers $r,s$ and two unimodular
      matrix polynomials $E(\x)$ and $F(\x)$ of appropriate dimensions
      such that 
      \[
        E(\x) \begin{bmatrix}
         I_r \\
           & P(\x) \\
        \end{bmatrix} F(\x) = 
        \begin{bmatrix}
          I_s \\
           & Q(\x) \\
        \end{bmatrix}. 
      \] 
    \end{definition}
    
    \begin{definition}[Linearization] \label{def:linearization}
      A linear matrix polynomial $L(\x)$ is a \emph{linearization} for 
      a matrix polynomial $P(\x)$ if $P(\x)$ is extended
      unimodularly equivalent to $L(\x)$. 
    \end{definition}
    
    In order to preserve the complete eigenstructure of a matrix polynomial, 
    it is of interest to maintain also the infinite eigenvalues, 
    which are defined as the zero eigenvalues of the 
    reversed polynomial. To achieve this we have to extend
    Definition~\ref{def:linearization}.
    
    \begin{definition}[Spectral equivalence]
      Two matrix polynomials $P(\x)$ and $Q(\x)$
      are \emph{spectrally equivalent} if $P(\x)$ is extended
      unimodularly equivalent to $Q(\x)$ and 
      $\rev P(\x)$ is extended unimodularly equivalent to $\rev Q(\x)$.
    \end{definition}
    
    \begin{definition}[Strong linearization] \label{def:strong-linearization}
	    A linear matrix polynomial $L(\x)$ is said to be a \emph{strong linearization} for 
	    a matrix polynomial $P(\x)$ if it
	    is spectrally equivalent to $L(\x)$. 
    \end{definition}

    \subsection{Working with product families of polynomials}
    
    The linearizations that we build in this work concern
    polynomials expressed as linear combinations of elements
    of a product family. Let us add more details about this concept. 
    
    With the term \emph{family of polynomials} (or \emph{polynomial family}) 
    we mean
    any set of elements in  $\mathbb F[\x]$ indexed
    on a finite totally ordered set $(I, \leq)$. 
    To denote these objects we use the notation 
    $\{ \phi_i(\x) \ | \ i \in I \}$ or 
    its more compressed form $\{ \phi_i(\x) \}$ or 
    even $\{ \phi_i \}$ whenever
    the index set $I$ and the variable $\x$ are clear from the context.
    Often the set $I$ will be a segment of the natural
    numbers or a subset of $\mathbb N^d$ endowed with
    the lexicographical order, as in
    Definition~\ref{def:productfamily}. 
    
    An important example of such families are
    the polynomials $\phi_i(\x)$ forming
    a basis for the polynomials of degree up to $k$. 
    Another extension that deserves our attention is
    the following. 
    
   \begin{definition} \label{def:productfamily}
     Given two families of polynomials
     $\{ \phi_i \}$ for $i = 0, \ldots, \epsilon$ and $\{ \psi_j \}$
     for $j = 0, \ldots, \eta$, we define the \emph{product family} as 
     the indexed set defined by:
     \[
       \phi \pf \psi := \{ \phi_i(\x) \psi_j(\x), \ i = 0, \ldots, \epsilon, \ 
          j = 0, \ldots, \eta \}. 
     \]
     with the lexicographical order (so that $(i,j) \leq (i', j')$ if
     either
     $i < i'$ or $i = i'$ and $j \leq j'$). 
   \end{definition}

   We introduce some notation that will make it easier in the following
   to deal with these product families and their use in linearizations.    
   We use the symbol $\pi_{k, \phi}(\x)$ to denote the column vector 
   \[
     \pi_{k, \phi}(\x) := \begin{bmatrix}
       \phi_k(\x) \\
       \vdots \\
       \phi_0(\x) \\
     \end{bmatrix}. 
   \]
   We will often identify $\pi_{k, \phi}(\x)$ with the family 
   $\{ \phi_i \ | \ i = 0, \ldots, k \}$ since they are just different representations of the same
   mathematical object. 
      
   Notice that Definition~\ref{def:productfamily} is easily extendable 
   to the product of an arbitrary number of families. In this case we always
   consider the lexicographical order on the new family, which is particular
   convenient because then we have 
   \[
     \pi_{k, \phi^{(1)} \pf \ldots \pf \phi^{(j)}}(\x) = 
       \pi_{\epsilon_1, \phi^{(1)}}(\x) \otimes 
       \ldots \otimes 
       \pi_{\epsilon_j, \phi^{(j)}}(\x). 
   \]
   
   \begin{remark} \label{rem:basis-map}
     Whenever the family $\{ \phi_i \}$ is a basis for the polynomials of
     degree at most $k$, every polynomial $p(\x) \in \mathbb F_k[\x]$
     can be expressed as 
     \[
       p(\x) = \sum_{j = 0}^k a_j \phi_j(\x). 
     \]
     In particular, the scalar product with $\pi_{k, \phi}(\x)$ is a
     linear isomorphism
     between $\mathbb{F}^{k+1}$ and the vector space of polynomials of degree at
     most $k$. We have
     \[
       \begin{array}{cccc}
       \Gamma_\phi: & \mathbb F^{k+1} & \longrightarrow & \mathbb F_k[\x] \\[5pt]
                    &     a           & \longmapsto     & \Gamma_\phi(a) := a^T \pi_{k, \phi}(\x)
       \end{array}. 
     \]
   \end{remark}
   
   With the above notation $\Gamma_\phi^{-1}(p(\x))$ is  the vector 
   of coordinates of $p(\x)$ expressed in the basis $\{ \phi_i \}$. 
   
   We recall the following definitions that can be found 
   in \cite{forney1975minimal}. 
   
   \begin{definition}
     A matrix polynomial $G(\x) \in \mathbb F[\x]^{k \times n}$ is a \emph{polynomial basis} 
     if its rows are a basis for a subspace of 
     the vector space of polynomial $n$-tuples. 
   \end{definition}
   
   \begin{definition}[Dual basis]
     Two polynomial bases $G(\x) \in \mathbb F^{k \times n}$ 
     and $H(\x) \in \mathbb F^{j \times n}$ are \emph{dual} if 
     $G(\x) H(\x)^T = 0$ and $j + k = n$. 
   \end{definition}
   
   We are interested in a particular subclass of dual bases which
   are relevant for our construction. We will call them 
   \emph{dual linear bases}. 
   
   \begin{definition}[Full row-rank linear dual basis] 
     We say that a $k \times (k+1)$ 
     matrix pencil $L_{k, \phi}(\x)$
     is a \emph{full row-rank linear dual basis}
     to $\pi_{k,\phi}(\x)$ (or, analogously, for a polynomial family $\{ \phi_i \}$) if
     $
       L_{k, \phi}(\x) \pi_{k, \phi}(\x) = 0, 
     $
     and $L_{k, \phi}(\x)$ has full row rank for any $\x \in \mathbb F$. 
   \end{definition}    
   
   Often we just say that $L_{k, \phi}(\x)$ is a 
   full row-rank linear dual basis, 
   meaning that it 
   is dual to $\pi_{k, \phi}(\x)$. Since the family $\{ \phi_i \}$
   is reported in the notation that we use for $L_{k, \phi}(\x)$, 
   there is no risk of ambiguity. 
   In the context of developing strong linearizations, we also give
   the following definition (which can again be found in 
   \cite{forney1975minimal}): 
   
   \begin{definition}[Minimal basis]
     A basis $G(\x) \in \mathbb F^{k \times n}$ 
     is said to be \emph{minimal} if the sum of degrees of its
     rows is minimal among all the possible bases of the
     vector space that they span. 
   \end{definition}
   
   We are particularly interested in \emph{dual minimal bases}, that 
   is bases that are both minimal and dual bases.    
   In \cite{forney1975minimal} it is shown that this 
   is equivalent to asking that $G(\x)$ and $H(\x)$ are of full
   row rank for any $\x \in \mathbb F$ and the same holds for
   the matrices with rows equal to 
   the highest degree coefficient of every row
   of $G(\x)$ and $H(\x)$. 
   When the leading coefficients of $G(\x)$ and $H(\x)$ have only
   nonzero rows this corresponds to
   their leading coefficient (and the condition is analogous to asking that 
   they are of full row rank for every $\x \in \overline{\mathbb F}$). 
   
   \begin{remark}   
	In the rest of the paper we will often consider full row-rank 
	linear dual bases
	(which will sometimes be minimal) related to polynomial families
	$\{ \phi_i \}$. In order to make the exposition simpler we will 
	call these bases \emph{dual}, without adding the term \emph{linear}
	and \emph{full row-rank}.
	However, 
	it must be noted that these are a very particular kind of dual bases
	and most of the results could not hold in a more general context. 
   \end{remark}

    \subsection{Building linearizations using product families}
    
    Let $P(\x)$ be a polynomial (or a matrix polynomial) expressed
    as a linear combination of elements of a product family
    $\phi \pf \psi$. In this section we provide a way of
    linearizing it starting from the coefficients of this representation. 
    In order to obtain this construction
    we rely on the
    following extension of the main result of 
    \cite{ldvdp}. 
    
  \begin{theorem} \label{thm:twobases}
     Let $L_{k,\phi}(\x), L_{k,\psi}(\x) \in \mathbb{C}^{k \times (k+1)}[\x]$
     be dual linear bases for two polynomial 
     families $\{ \phi_i \}$ and $\{ \psi_i \}$. Assume that the elements
     of each polynomial family have no common divisor, that is there
     exists a vector $w_{k,\star}$ such that 
     $\pi_{k,\star}(\x)^T w_{k, \star} = 1$ 
     for $\star \in \{ \phi, \psi \}$. 
     Then the matrix polynomial 
     \[
       \mathcal L(\x) := \begin{bmatrix}
         \x M_1 + M_0 & L_{\eta,\phi}(\x)^T \otimes I \\
         L_{\epsilon,\psi}(\x) \otimes I & 0 \\
       \end{bmatrix}
     \]
     is a linearization for $P(\x) = (\pi_{\eta,\phi}(\x) \otimes I)^T 
     (\x M_1 + M_0) (\pi_{\epsilon,\psi}(\x) \otimes I)$,
     which is a polynomial expressed in the product family
     $\phi \pf \psi$. Moreover, 
     this linearization is strong\footnote{%
     Notice that the linearization is guaranteed to be strong for the
     matrix
     polynomial formally defined by $(\pi_{\eta,\phi}(\x) \otimes I)^T 
          (\x M_1 + M_0) (\pi_{\epsilon,\psi}(\x) \otimes I)$. In particular, 
          this expression might provide a matrix polynomial with leading
          coefficient zero, but we still need to consider that polynomial
          and not the one with the leading zero coefficients removed, 
          otherwise the strongness might be lost. 
     }
     if the dual bases are minimal and have reversals of full row rank. 
   \end{theorem}
   
   \begin{proof}
     We mainly follow the proof given in \cite{ldvdp}. 
     Recall that we can find $b_{k,\star}$ such
     that the matrix polynomial
     \[
       S_{k,\star}(\x) := \begin{bmatrix}
         L_{k,\star}(\x) \\ b_{k,\star}^T \\
       \end{bmatrix}
     \]
     is unimodular \cite{beelen1988pencil}, and we know that 
     $S_{k,\star}(\x) \pi_{k,\star}(\x) = \alpha_{k,\star}(\x) e_{k+1}$. 
     This can be rewritten as 
     $\pi_{k,\star}(\x) = \alpha_{k,\star}(\x) S_{k,\star}^{-1}(\x)  e_{k+1}$. Since the entries of 
     $\pi_{k,\star}(\x)$ do not have any common factor we
     conclude that $\alpha_{k,\star}(\x)$ is a
     nonzero constant (and so we can drop the 
     dependency on $\x$). 
     We remark that rescaling the vector $b_{k,\star}$ by 
     a non-zero constant preserves the unimodularity of
     $S_{k,\star}(\x)$ (since it is equivalent to left multiplying
     by an invertible diagonal matrix). For this reason we can assume
     that $b_{k,\star}$ is chosen so that $S_{k,\star}(\x) \pi_{k,\star}(\x) = e_{k+1}$. 
     We define
     $
       V_{k,\star}(\x) := S_{k,\star}(\x)^{-1}
     $
     so that $V_{k,\star}(\x) e_{k+1} = \pi_{k,\star}(\x)$. 
     With these hypotheses we have that
     
     \[
       L_{k,\star}(\x) V_{k,\star}(\x) = \begin{bmatrix} I & 0 \end{bmatrix},  \qquad 
       V_{k,\star}^T(\x) L_{k,\star}^T(\x) = \begin{bmatrix}
         I \\
         0 \\
       \end{bmatrix}
     \]
     Now observe that the matrix polynomial $\mathcal L(\x)$ 
     can be transformed by means of a unimodular 
     transformation in the following way:
     \[
       \begin{bmatrix} 
         V^T_{\eta,\phi}(\x) \otimes I & X(\x) \\
           0            & I     \\
       \end{bmatrix}
       \begin{bmatrix}
         \x M_1 + M_0 & L_{\eta,\phi}^T(\x) \otimes I \\
         L_{\epsilon,\psi}(\x) \otimes I & 0 \\
       \end{bmatrix}
       \begin{bmatrix}
         V_{\epsilon,\psi}(\x) \otimes I & 0 \\
         Y(\x)                 & I \\
       \end{bmatrix} =: \tilde P(\x)
     \]
     where $\tilde P(\x)$ can be chosen as follows:
     \[
       \tilde P(\x) := \begin{bmatrix}
         0 & 0 & I \\
         0 & P(\x) & 0 \\
         I & 0 & 0 \\
       \end{bmatrix}, \quad 
       P(\x) = (\pi_{\eta, \phi}(\x) \otimes I)^T 
              (\x M_1 + M_0) (\pi_{\epsilon, \psi}(\x) \otimes I).
     \]
     The matrices $X(\x)$ and $Y(\x)$ can be chosen in order to 
     put zeros in the top-left corner almost everywhere, and 
     the only non zero block $P(\x)$ can be retrieved
     by knowing $V_{k, \star}(\x) e_k$, as usual for 
     $\star \in \{ \phi, \psi \}$. 
     
     We now check that the linearization is strong. 
     Similarly to the previous step, we can find a 
     constant vector $u_{k,\star}$ such that 
     \[
       \tilde S_{k,\star}(\x) = 
       \begin{bmatrix}
         u_{k,\star}\\
         \rev L_{k,\star}(\x)
       \end{bmatrix}, \qquad 
       \star \in \{ \phi, \psi \}
     \]
     and $\tilde S_{k,\star}(\x) \rev \pi(\x) = \tilde\alpha_{k,\star}(\x) e_1$. Since the entries of $\pi_{k,\star}(\x)$ do not share
     any common factor, we get that $\tilde\alpha_{k,\star}$ 
     is a nonzero constant. As in the previous case, 
     applying a diagonal scaling does not change the unimodularity so
     we can assume that $\tilde\alpha_{k,\star} = 1$. Define
     $W_{k,\star}(\x) = \tilde S_{k,\star}(\x)^{-1}$
     so that $W_{k,\star}(\x) e_1 = \rev \pi_{k,\star}(\x)$. 
     
     We can perform another unimodular transformation
     on the reversed polynomial. Let $A(\x)$ be defined
     as follows:
     \[
      \begin{bmatrix}
        W^T_{\eta,\phi}(\x) \otimes I & \hat X(\x) \\
        0                 & I \\
      \end{bmatrix}
      \begin{bmatrix}
        M_1 + \x M_0 & \rev L_{\eta,\phi}^T(\x) \otimes I \\
        \rev L_{\epsilon,\psi}(\x) \otimes I & 0 \\
      \end{bmatrix}
      \begin{bmatrix}
        W_{\epsilon,\psi}(\x) \otimes I & 0 \\
        \hat Y(\x) & I \\
      \end{bmatrix}. 
     \]
     Notice that $\rev L_{k,\star}(\x) W_{k,\star}(\x)
     = [ 0 \ I ]$ so we can write
     \[
       A(\x) = \begin{bmatrix}
         A_{1,1}(\x) & 0 & 0 \\
         0 & 0 & I \\
         0 & I & 0 \\
       \end{bmatrix}
     \]
     by appropriately choosing $\hat X(\x)$ and
     $\hat Y(\x)$. In particular we have
     \[
       A_{1,1}(\x) = (\rev \pi_{\eta,\phi}(\x)\otimes I)^T (M_1 + \x M_0) (\rev \pi_{\epsilon,\psi}(\x) \otimes I) = 
     \rev P(\x)\] if the degree of $P(\lambda)$ is maximum (i.e., if the
     coefficient that goes in front of the maximum degree term
     in the previous relation is not zero). 
   \end{proof}    
    
  \section{Further extensions of linearizations}
  
  In this section we will show some concrete examples for the choice
  of the product families $\phi \pf \psi$ and some related applications. 
  In particular, we show that this apparently abstract way of rewriting
  polynomials is useful in many different situations in order to 
  solve some problems directly, without unneeded change of bases.     
    
      \label{sec:lin-productbases}
    
  The most natural problem
  to consider is
  the case where both $\{ \phi_i \}$ and $\{ \psi_i \}$
  are polynomial bases for the polynomials of degree up to $\epsilon$
  and $\eta$, respectively. 
  
  We recall that in this case the product family $\phi \pf \psi$ is, 
  in general, not a basis since it contains too many vectors for the
  dimension of the vector space that it needs to span. 
  However, this extra flexibility in
  the representation will be useful in many situations.   
    
    \subsection{An extension to more than two bases}
    \label{sec:morebases}
    
    Given the above formulation for a linearization of a polynomial
    expressed in a product family, it is natural to ask if the framework can
    be extended to cover more than two bases, 
    that is to product families of the form
    \[
      \phi^{(1)} \pf \ldots \pf \phi^{(j)} := \{  
        \phi^{(1)}_{i_1} \ldots \phi^{(j)}_{i_j} \ | \  
        i_s = 0, \ldots, k_s, \ s = 1, \ldots, j
      \}
    \]
    where $\{ \phi^{(s)}_i \ | \ i = 0, \ldots, k_s \}$ are families of
    polynomials for $s = 1, \ldots, j$. 
    
    We show that there is no need to extend Theorem~\ref{thm:twobases}, 
    but it is sufficient to construct two appropriate
    dual bases $L_{\epsilon, \phi}(\x)$
    and $L_{\eta, \psi}(\x)$ to deal with this case. We only need
    to prove that  the hypotheses 
    of Theorem~\ref{thm:twobases} are satisfied. 
    
    \begin{definition}\label{def:productdualbasis}
      Let $L_{\epsilon, \phi}(\x)$ and $L_{\eta,\psi}(\x)$ be two dual
      bases
      for two families $\{ \phi_i \}$ and $\{ \psi_i \}$. Let  
      $w$ be a constant vector such that $w^T \pi_{\eta, \psi}(\x)$
      is a nonzero constant, 
      and $A$ an invertible matrix. 
      We say that the matrix 
      \[
        L_{k, \phi \pf \psi}(\x) = \begin{bmatrix}
          A \otimes L_{\eta, \psi}(\x) \\
          L_{\epsilon, \phi}(\x) \otimes w^T \\
        \end{bmatrix}, \qquad 
        k := (\epsilon+1) (\eta + 1) - 1, 
      \]
      is a \emph{product dual basis} of $L_{\epsilon, \phi}(\x)$ and $L_{\eta, \psi}(\x)$. We denote
      it as $L_{\epsilon, \phi}(\x) \prm L_{\eta, \psi}(\x)$. 
    \end{definition}
    
    Notice that, since the product
    dual basis is not unique, the previous notation
    actually denotes a family of such matrices so 
    we should be writing $L_{k, \phi \pf \psi}(\x) \in 
    L_{\epsilon,\phi}(\x) \prm L_{\eta, \psi}(\x)$. However, 
    in the following we will often write, by slight abuse
    of notation,  
    $L_{k, \phi \pf \psi}(\x) =
        L_{\epsilon,\phi}(\x) \prm L_{\eta, \psi}(\x)$.
    
    The above definition can be extended easily to a product
    of arbitrary families, by means of the following.  
    
    \begin{definition} 
      We say that, for any families of polynomials 
      $\{ \phi_i^{(1)} \}, \ldots, \{ \phi_i^{(j)} \}$,
      the matrix $L_{k, \phi^{(1)} \pf \ldots \pf \phi^{(j)}}(\x)$ 
      is a \emph{product dual basis} for these families, and 
      we denote it as $L_{\epsilon_1, \phi^{(1)}} \prm \ldots \prm 
      L_{\epsilon_j, \phi^{(j)}}(\x)$, 
      where
      \[
        L_{k, \phi^{(1)} \pf \ldots \pf \phi^{(j)}}(\x) = 
          ( L_{\epsilon_1, \phi^{(1)}}(\x) \prm \ldots \prm L_{\epsilon_{j-1}, \phi^{(j-1)}}(\x)) \prm L_{\phi^{(j)}}(\x). 
      \]
    \end{definition}
    
    Notice that the above formula provides a recursive manner
    for computing such product dual bases. 
    In the next lemma we show that they can be used to 
    construct linearizations in the spirit of
    Theorem~\ref{thm:twobases}. 
    
    \begin{lemma} \label{lem:productfamilyfullrowrank}
      Let $L_{k, \phi \pf \psi}(\x) = L_{\epsilon, \phi}(\x) \prm L_{\eta, \psi}(\x)$ 
      be a product dual basis. Then
      \begin{enumerate}[(i)]
              \item If $\pi_{k, \phi \pf \psi}(\x)$ is the vector containing the elements of 
              the product family
                $\phi \pf \psi$,  then
               $L_{k, \phi \pf \psi}(\x) \pi_{k, \phi \pf \psi}(\x) = 0. $  
        \item $L_{k, \phi \pf \psi}(\x)$ is a rectangular matrix with full row rank and size $k \times (k+1)$ for all values of 
        $\x$. 
      \end{enumerate}
    \end{lemma}
    
    \begin{proof}
      We first check condition $(i)$. Notice that  we have
      $
        \pi_{k, \phi \pf \psi}(\x) = \pi_{\epsilon, \phi}(\x) \otimes 
      \pi_{\eta, \psi}(\x),
      $
      according to the ordering specified in
      Definition~\ref{def:productfamily}. For this reason
      we can write
      \[
        L_{k, \phi \pf \psi}(\x) \pi_{k, \phi \pf \psi}(\x)
        = \begin{bmatrix}
          A \pi_{\epsilon, \phi}(\x) \otimes L_{\eta, \psi}(\x) \pi_{\eta, \psi}(\x)\\
          L_{\epsilon, \phi}(\x) \pi_{\epsilon, \phi}(\x) \otimes w^T \pi_{\eta, \psi}(\x) \\
      \end{bmatrix} = 0. 
      \]
      The number of rows and columns can be verified by
      direct inspection in Definition~\ref{def:productdualbasis}. 
      Concerning condition $(ii)$ we shall check that, for any $\x$, 
      the only vectors in the right kernel of $L_{k, \phi \pf \psi}(\x)$
      are multiples of 
      $\pi_{k, \phi \pf \psi}(\x)$. Let $v(\x)$ be such a vector, 
      so that $L_{k, \phi \pf \psi}(\x) v(\x) = 0$. We can partition
      $v(\x) = [ v_0(\x) \ \dots \ v_{\epsilon}(\x)]^T$ according to 
      the Kronecker structure of $L_{k, \phi \pf \psi}(\x)$ so,
      recalling that $A$ is invertible, we
      have 
      \[
        L_{k, \phi \pf \psi}(\x) v(\x) = 0 \iff 
          \begin{cases}
            L_{\eta, \psi}(\x) v_j(\x) = 0  \\ 
            (L_{\epsilon, \phi}(\x) \otimes w^T) v(\x) = 0\\
          \end{cases} j = 0, \ldots, \epsilon. 
      \]
      The first relation tells us that $v_j(\x) = \alpha_j(\x) \pi_{\eta, \psi}(\x)$, due to $L_{\eta, \psi}(\x)$ being of full row rank. 
      If 
      we set $\alpha(\x) = [ \alpha_0(\x) \ \ldots \ \alpha_\epsilon(\x) ]^T$
      we have $v(\x) = \alpha(\x) \otimes \pi_{\eta, \psi}(\x)$, so that 
      the last equation becomes $L_{\epsilon, \phi}(\x) \alpha(\x) \otimes 
      w^T \pi_{\eta, \psi}(\x) = 0$. Since 
      $w^T \pi_{\eta, \psi}(\x) \neq 0$, the only solution is given by 
      $\alpha(\x) = \pi_{\epsilon, \phi}(\x)$. 
    \end{proof}
    
    \begin{remark} \label{rem:productfamilyconstruction}
      The proof of Lemma~\ref{lem:productfamilyfullrowrank} shows that this
      construction is not the only possible one. As an immediate example, 
      we could have defined $L_{\epsilon, \phi}(\x) \prm L_{\eta, \psi}(\x)$
      to be the matrix
      \[
        \tilde L_{k, \phi \pf \psi}(\x) = \begin{bmatrix}
          w^T \otimes L_{\eta, \psi}(\x) \\
          L_{\epsilon, \phi}(\x) \otimes A \\
        \end{bmatrix}, \qquad 
        k := (\epsilon+1) (\eta + 1) - 1,         
      \]
      with the same hypotheses of Definition~\ref{def:productdualbasis}, 
      and the proof would have been essentially the same. 
    \end{remark}
    
    \begin{remark}
    Lemma~\ref{lem:productfamilyfullrowrank}
    justifies the notation $L_{k, \phi \pf \psi}(\x)$ that we have used until now, since the
    product dual basis is a dual basis for the 
    product family $\phi \pf \psi$.  
    \end{remark}
    
    \begin{remark}
      Given the structure of the matrix $L_{\epsilon,\phi} \prm L_{\eta,\psi}(\x)$ that we
      have defined above it might be natural to ask if the more
      general matrix
      \[
        M(\x) = \begin{bmatrix}
        A \otimes L_{\eta, \psi}(\x) \\
         L_{\epsilon, \phi}(\x)  \otimes B  \\
        \end{bmatrix}, \qquad 
        A \in \mathbb{C}^{k_1 \times (\eta+1)}, \
        B \in \mathbb{C}^{k_2 \times (\epsilon + 1)}
      \]
      and such that $k_1 \eta + k_2 \epsilon = (\eta + 1)(\epsilon + 1) - 1$
      can be a product dual basis when $A$ and $B$ are of full row-rank. 
      The answer is no unless $k_1 = \epsilon + 1$ or $k_2 = \eta + 1$ 
      and so we are again back
      in the above two cases, as the next lemma shows. 
    \end{remark}
    
    \begin{lemma}
      Let $M(\x)$ be a matrix of the form
      \[
        M(\x) = \begin{bmatrix}
        A \otimes L_{\eta, \psi}(\x) \\
         L_{\epsilon, \phi}(\x)  \otimes B  \\
        \end{bmatrix}, \qquad 
        A \in \mathbb{C}^{k_1 \times (\epsilon+1)}, \
        B \in \mathbb{C}^{k_2 \times (\eta + 1)}
      \]
      with $L_{\epsilon, \phi}(\x)$ and $L_{\eta, \psi}(\x)$ 
      dual bases for $\pi_{\epsilon,\phi}(\x)$ and 
      $\pi_{\eta,\psi}(\x)$, $A$ and $B$ of full row rank with $k_1$ and
      $k_2$ rows and $\epsilon+1$ and $\eta+1$ columns, respectively. 
      Then the right kernel of $M(\x)$ has dimension
      at least $1 + (\epsilon + 1 - k_1)(\eta + 1 - k_2)$ for at least
      one value of $\x$ if either $A \pi_{\epsilon,\phi}(\x) \not\equiv 0$ or 
      $B \pi_{\eta,\psi}(\x) \not\equiv 0$. 
    \end{lemma}
    
    \begin{proof}
      Let $S_A = \{ v \ | \ Av = 0\}$ and $S_B = \{ w \ | \ Bw = 0 \}$
      be the right kernels of $A$ and $B$ which have dimensions
      $(\epsilon + 1 - k_1)$ and $(\eta + 1 - k_2)$, respectively. 
      We have that the span of 
      $\pi_{\epsilon,\phi}(\x) \otimes \pi_{\eta,\psi}(\x)$
      and $S_A \otimes S_B$ are included in the kernel
      of $M(\x)$. Since $A \pi_{\epsilon, \phi}(\x) \neq 0$ (or,
      analogously, $B \pi_{\eta, \psi}(\x) \neq 0$) for
      at least one $\x$ we have that the dimension of the 
      union of these two spaces is at least $1 + (\epsilon + 1 - k_1)(\eta + 1 - k_2)$, which concludes the proof. 
    \end{proof}
    
   Lemma~\ref{lem:productfamilyfullrowrank} can be 
    generalized to the product of more families of 
    polynomials, yielding the following. 
    
    \begin{corollary} \label{cor:productfamilyfullrowrankmany}
      Let $L_{\epsilon_1, \phi^{(1)}}(\x) \prm \ldots \prm L_{\epsilon_j, \phi^{(j)}}(\x)$ be a product dual basis of 
      $j$ dual bases. Then it has full row rank and the only
      elements in its right kernel are multiples of 
      $\pi_{k, \phi^{(1)} \pf \ldots \pf \phi^{(j)}(\x)}$, 
      independently of the construction chosen (either the one
      of Lemma~\ref{lem:productfamilyfullrowrank}
      or Remark~\ref{rem:productfamilyconstruction})
      
    \end{corollary}
    
    \begin{proof}
      Exploit the recursive definition of $L_{\epsilon_1, \phi^{(1)}}(\x) \prm \ldots \prm L_{\epsilon_j, \phi^{(j)}}(\x)$ and 
      apply Lemma~\ref{lem:productfamilyfullrowrank}. 
    \end{proof}
       
    The construction of these product dual bases allows
    us to formulate the following result, which can be seen
    as an extension of Theorem~\ref{thm:twobases} that makes
    it possible to handle more than two bases at once. 
    
    \begin{theorem} \label{thm:morebases}
      Let $\{ \phi_i^{(1)} \}, \ldots, \{ \phi_i^{(j)} \}$ and 
      $\{ \psi_i^{(1)} \}, \ldots, \{ \psi_i^{(l)} \}$ be 
      families of polynomials. Then the matrix polynomial
      \[
        \mathcal L(\x) = \begin{bmatrix}
          \x M_1 + M_0 & (L_{\epsilon_1, \phi^{(1)}} \prm \ldots \prm L_{\epsilon_j, \phi^{(j)}}(\x))^T \\
          L_{\eta_1, \psi^{(1)}} \prm \ldots \prm L_{\eta_l, \psi^{(l)}}(\x) & 0 \\
        \end{bmatrix}
      \]
      is a linearization for the polynomial
      \[
        P(\x) = (\pi_{\epsilon_1, \phi^{(1)}}(\x) \otimes \ldots \otimes \pi_{\epsilon_j, \phi^{(j)}}(\x))^T 
      (\x M_1 + M_0) (\pi_{\eta_1, \psi^{(1)}}(\x) \otimes \ldots \otimes \pi_{\eta_l, \psi^{(l)}}(\x)). 
      \]
    \end{theorem} 
    
    \begin{proof}
      Apply Theorem~\ref{thm:twobases}, whose hypothesis are satisfied because
      of Lemma~\ref{lem:productfamilyfullrowrank} and 
      Corollary~\ref{cor:productfamilyfullrowrankmany}. 
    \end{proof}
    
    Here is an example of the structure that the matrix 
    $L_{\epsilon, \phi} \prm L_{\eta, \psi}(\x)$ can have in a
    simple case. Let $\{ \phi_i \}$ 
    be the Chebyshev basis, while the family $\{ \psi_i \}$ 
    is any degree graded polynomial family. The matrix 
    $L_{\epsilon, \phi} \prm L_{\eta, \psi}(\x)$ can be realized by
    the following by choosing $A = I$ and $w = e_{\eta + 1}$. 
    \[
    \begingroup
    \newcommand{\eone}[1]{\begin{array}{ccc} & & #1  \end{array}}
      L_{\epsilon, \phi} \prm L_{\eta, \psi}(\x) = \begin{bmatrix}
        L_{\eta, \psi}(\x) \\
        & L_{\eta, \psi}(\x) \\
        && \ddots \\
        &&& L_{\eta, \psi}(\x) \\
        &&&& L_{\eta, \psi}(\x) \\
        \eone 1 & \eone{-2\x} & \eone 1  \\
        & \ddots & \ddots & \ddots \\
        &&  \eone 1 & \eone{-2\x} & \eone 1  \\
        &&& \eone 1 & \eone{-\x} \\
      \end{bmatrix}.
      \endgroup
    \]  
    
    In order to give an example of how these variations behave 
    in practice, we consider what happens when taking the product basis
    of several monomial bases.
    
    The monomial basis, in this setting, is rather special. In fact, 
    the elements of the product family of two monomial bases are of
    the form $\lambda^i \lambda^j = \lambda^{i+j}$ and 
    so they correspond to elements of a (larger) monomial basis. 
    However, notice that this is 
    not true in general, as for example
    when considering $\phi_i(\x)$ belonging to
    other polynomial bases. 
     
    We can exploit this fact by rephrasing any polynomial expressed
    in the monomial basis as a polynomial in the product family of
    two monomial bases (like in \cite{ldvdp}) or also in the product
    family of more bases, by using the framework above. 
      
      We show here some examples to illustrate
      some possibilities. 
      Let $p(\x) = \sum_{i = 0}^3 p_i \lambda^i$ a
      degree $3$ polynomial. Then we can obtain different linearizations for it. 
      
      As a first example, choosing
      $\{ \psi_i \} = \{ 1, \lambda, \lambda^2 \}$ 
      and $\{ \phi_i \} = \{ 1 \}$ yields the classical
      Frobenius form: 
      \[
        \mathcal L(\x) = \left[ \begin{array}{ccc}
          \x p_3 + p_2 & p_1 & p_0 \\ \hline
          1 & -\x \\
          & 1 & -\x \\
        \end{array} \right]. 
      \]
      We can instead choose $\{ \psi_i \} = \{ \phi_i \} = \{ 1, \x \}$ 
      and obtain a symmetric linearization
      (this is only one of the possibilities for distributing the coefficients): 
      \[
        \mathcal L(\x) = \left[\begin{array}{cc|c} 
          \x p_3 + p_2 & \frac 1 2 p_1 & 1 \\
          \frac 1 2 p_1 & p_0 & -\x \\\hline
          1 & -\x & 0 \\
        \end{array}\right]. 
      \]
      But we can also choose to set $\{ \psi_i \} = \{ 1, \x \} \pf \{ 1, \x \}$ and $\{ \phi_i \} = \{ 1 \}$, and we obtain:
      \[
        \mathcal L(\x) = \left[\begin{array}{cccc} 
          \x p_3 + p_2 & \frac 1 2 p_1 & \frac 1 2 p_1 & p_0 \\ \hline
          1 & -\x \\ 
          && 1 & -\x \\ \hline
          & 1 & & -\x \\
        \end{array}\right]. 
      \]
      One thing can be noticed immediately: we have increased the dimension
      of the problem. In fact the matrix $L(\x)$ that we have used in the lower part 
      has its dimension increased by $1$ since it represents $\x$
      two times. 
      This has the consequence that while $L_{1,1,\phi}(\x)$ has full row rank its reversal
      does not, and so the linearization is not strong. In fact, here we have
      a spurious infinite eigenvalue. 
      
    \subsection{Handling orthogonal bases}
    
    This section is devoted to study the different structure of the
    dual basis $L_{k, \phi}(\x)$ when $\{ \phi_i \}$ is a 
    non-monomial basis. 
    
    We first deal with the case where the basis
    $\{ \phi_i(\x) \}$ is degree graded 
   and satisfies a three-terms recurrence relation of the form
   \begin{equation} \label{eq:recurrence}
     \alpha \phi_{j+1}(\x) = (\x-\beta) \phi_j(\x) - \gamma \phi_{j-1}(\x), 
     \qquad \alpha \neq 0, \quad j > 0, 
   \end{equation}    
   which includes all the orthogonal polynomials with a 
   constant three term recurrence (with the possible
   exception of the first two elements of the basis). 
   Notice, however, that the result
   can be easily generalized to more general recurrences. 
   \begin{lemma} \label{lem:recurrence-relation}
     Let $\{ \phi_i \}$ be a degree graded basis satisfying the
     three-terms recurrence
     relation~\eqref{eq:recurrence}. Then the linear matrix polynomial 
     $L_{k,\phi}(\x)$ of size $k \times (k+1)$ defined as follows
     \[
       L_{k, \phi}(\x) := \begin{bmatrix}
              \alpha & (\beta - \x) & \gamma \\
              & \ddots & \ddots & \ddots \\
              & & \alpha & (\beta - \x) & \gamma \\
              & & & \phi_0(\x) & -\phi_1(\x) \\
            \end{bmatrix}
     \]
     has full row rank for any $\x \in \mathbb F$ 
     and is such that 
     \[
       L_{k, \phi}(\x) \pi_{k, \phi}(\x) = 0, \qquad \text{with }  
       \pi_{k,\phi}(\x) := \begin{bmatrix}
         \phi_{k}(\x) \\
         \vdots \\
         \phi_0(\x) \\
       \end{bmatrix}. 
     \]
     Moreover, the leading coefficient of $L_{k,\phi}(\x)$ 
     has full row rank. 
   \end{lemma}
   
   \begin{proof}
     It is immediate to verify that $L_{k, \phi}(\x) \pi_{k, \phi}(\x) = 0$, since
     each row of $L_{k,\phi}(\x)$ but the last one is just the
     recurrence relation of \eqref{eq:recurrence} and 
     the last one yields $\phi_0(\x) \phi_1(\x) - \phi_1(\x) \phi_0(\x) = 0$. 
     
     We can then check that the matrix has full row rank. Notice
     that the first $k$ columns of $L_{k,\phi}(\x)$ form an upper
     triangular matrix with determinant $\alpha^{k-1} \phi_0(\x)$.
     The basis is degree graded so $\phi_0(\x)$ is an invertible
     constant and $L_{k,\phi}(\x)$ contains an invertible matrix
     of order $k \times k$, thereby proving our claim. 
     
     It is immediate to verify the last claim, since the leading coefficient
     of $L_{k,\phi}(\x)$ with the first column removed is a diagonal matrix
     with nonzero elements on the diagonal, and so it is invertible. 
   \end{proof}
   
   We can immediately construct some examples for the application of the 
   theorem. Consider the Chebyshev basis of the first kind $\{ T_i(\x) \}$, which
   satisfies a recurrence relation of the form: 
   \[
     T_{j+1}(\x) = 2\x T_j(\x) - T_{j-1}(\x). 
   \]
   Then we have that the matrix polynomial 
   \[
     \mathcal L(\x) = \begin{bmatrix}
       \lambda M_1 + M_0 & L_{\eta,T}(\x)^T \\
       L_{\epsilon,T}(\x) & 0 \\
     \end{bmatrix}, \qquad 
     L_{k, T}(\x) := \begin{bmatrix}
       1 & -2\x & 1 \\
         & \ddots & \ddots & \ddots \\
         &        & 1      & -2\x    & 1 \\
         &        &        & 1      & -\x \\
     \end{bmatrix}
   \]
   is a linearization for the polynomial $p(\x) = \sum_{i = 1}^{\epsilon}
   \sum_{j = 1}^\eta (\lambda M_1 + M_0)_{i,j} T_i(\x) T_j(\x)$. As 
   shown in \cite{lawrence2016constructing}, the product $T_i(\x) T_j(\x)$ can be
   rephrased in terms of sums of Chebyshev polynomials, and 
   this can be used to build a linearization for polynomials
   expressed in the Chebyshev basis (so without product families involved).     
   
   \subsection{Handling interpolation bases}
   
   The framework covers orthogonal bases, but 
   there are some other interesting cases, as for example 
   the interpolation bases such as Lagrange, Newton and Hermite. 
   
   In this section we study their structures. Recall that, 
   by Theorem~\ref{thm:twobases}, once we have constructed
   the dual basis $L_{k, \phi}(\x)$ for one of these bases, 
   we need to ensure that $L_{k, \phi}(\x) \pi_{k,\phi}(\x) = 0$ and 
   that $L_{k, \phi}(\x)$ has full row rank. In order to have a strong
   linearization we also require the dual basis to be minimal. 
   
\subsection{The Lagrange basis}
  
  Let $\sigma_1^{(1)}$, $\ldots$, $\sigma^{(1)}_\epsilon$ and 
  $\sigma_1^{(2)}$, $\ldots$, $\sigma_\eta^{(2)}$ two
  (not necessarily disjoint) sets of pairwise different
  nodes in the complex plane. Then we 
  can define the weights and the Lagrange polynomials by
  \[
    t_i^{(s)} := \prod_{j \neq i} (\sigma_i^{(s)} - \sigma_j^{(s)}), \qquad 
    l_i^{(s)}(\x) := \frac{1}{t_i^{(s)}} \prod_{j \neq i} (\x - \sigma_j^{(s)}), 
    \qquad s \in \{ 1,2 \}. 
  \]
  In the following let $\phi_j(\x) = l_j^{(1)}(\x)$ 
  and $\psi_j(\x) = l_j^{(2)}(\x)$, coherently with the notation
  used before. 
  The linearization for a polynomial expressed in a product family,
  built according to Theorem~\ref{thm:twobases},
  has the following structure: 
  \[  
   \mathcal L(\x) = \begin{bmatrix}
     \x M_1 + M_0 & L_{\eta,\psi}(\x)^T \\
     L_{\epsilon,\phi}(\x) & 0 \\
   \end{bmatrix}
  \]
  where 
  \[
    L_{k,\phi}(\x) = \begin{bmatrix}
      t_1^{(1)} (\x - \sigma_1) & -t_2^{(1)} (\x - \sigma_2) \\
      & \ddots & \ddots \\
      & & t_{k-1}^{(1)} (\x- \sigma_{k-1}) & -t_k^{(1)} (\x - \sigma_k) \\
    \end{bmatrix}
  \]
  and $L_{k,\psi}(\x)$ can be defined in an analogous way. 
  
  \begin{lemma}
    The matrix $L_{k, \phi}(\x)$ defined above is a dual minimal 
    basis for the
    Lagrange basis $\{ \phi_i \}$ constructed on the nodes
    $\sigma_1, \ldots, \sigma_k$ (that is, it is dual to $\pi_{k,\phi}(\x)$). 
  \end{lemma}
  
  \begin{proof}
    It is easy to verify that $L_{k, \phi} \in \mathbb{C}[\x]^{k \times (k+1)}$
    and 
        \[
          L_{k,\phi}(\x) \pi_{k,\phi}(\x) = 0, \qquad 
          \pi_{k,\phi}(\x) := \begin{bmatrix}
            l_{k}^{(1)}(\x) \\
            \vdots \\
            l_{0}^{(1)}(\x) \\
          \end{bmatrix}. 
        \]
     It remains to show that the matrix $L_{k, \phi}(\x)$ has full row rank
     for any $\x \in \mathbb F$. 
     For all values of $\x$ that are not equal to the nodes
     the first $k$ columns are upper triangular with non-zero elements
     on the diagonal, and so the hypotheses is satisfied. It remains
     to deal with the cases where $\x = \sigma_i$ for some $i = 1, \ldots, k - 1$. 
     
     We note that in this case one of the columns of the matrix is zero, but removing it yields a square matrix which is block diagonal with only
     two diagonal blocks. The top-left one is upper triangular and invertible, 
     while the bottom-right one is lower triangular and invertible, since
     they both have nonzero elements on the diagonal.
     
     Notice that the first $k$ columns of the leading coefficient 
     of $L_{k,\phi}$ are upper triangular with nonzero elements on the diagonal. 
     This implies that the leading coefficient has full row rank, thus
     proving the minimality of $L_{k,\phi}(\x)$. 
  \end{proof}
    
  \subsection{Constructing a classical Lagrange linearization}
    \label{sec:lagrange-linearization}
    
  Besides building linearizations for polynomial
  expressed in product families of Lagrange bases, 
  the above formulation can be used to linearize a polynomial expressed
  in a Lagrange basis built on the union of the nodes. 
  
  In fact, we observe that if we have two Lagrange polynomials
  $l_i^{(1)}(\x)$ and $l_j^{(2)}(\x)$ defined according
  to the previous notation then their product
  is almost a Lagrange polynomial for the union of the nodes. 
  More
  precisely, assume that we have a set of nodes 
  $\sigma_1, \ldots, \sigma_n$ and let $l_i^{(1)}(\x)$ and $l_j^{(2)}(\x)$
  be Lagrange polynomials relative to the nodes $\sigma_1, \ldots, \sigma_{k}$
  and $\sigma_{k+1}, \ldots, \sigma_n$, respectively. 
  Then if $l_i(\x)$ are the
  Lagrange polynomials related to all the nodes we have that 
  \[
    l_i(\x) = \begin{cases}
      l_i^{(1)}(\x) \cdot l_{j}^{(2)}(\x) \cdot \frac{\x - \sigma_{j+k}}{\sigma_i - \sigma_{j+k}} \prod_{s \neq j} \frac{\sigma_{j+k} - \sigma_s}{\sigma_i - \sigma_{s+k}} & 
      i \leq k \\ 
      l_j^{(1)}(\x) \cdot l_{i-k_1}^{(2)}(\x) \cdot \frac{\x - \sigma_{j}}{\sigma_i - \sigma_{j}} \prod_{s \neq j} \frac{\sigma_{j} - \sigma_s}{\sigma_i - \sigma_{s}} & i > k \\ 
    \end{cases}. 
  \]
  It is worth noting that these formulas become much more straightforward
  if one considers unscaled Lagrange polynomials by getting rid of
  the normalization factor, since in that case we obtain: 
  \[
    l_i(\x) = \begin{cases}
      l_i^{(1)}(\x) \cdot l_{j}^{(2)}(\x) \cdot (\x - \sigma_{j+k}) & 
      i \leq k \\ 
      l_j^{(1)}(\x) \cdot l_{i-k_1}^{(2)}(\x) \cdot (\x - \sigma_{j})
      & i > k \\ 
    \end{cases}. 
  \]
  The part missing from the product of two Lagrange polynomials
  in order to obtain the one with the union of the nodes is always 
  linear and so can be placed as a coefficient in the top-left
  matrix polynomial $\x M_1 + M_0$. 
%
%
  \begin{remark}
    Notice that it is possible to choose two equal nodes in 
    $\sigma_1, \ldots, \sigma_k$ and $\sigma_{k+1}, \ldots, \sigma_n$. 
    This allows to obtain a Lagrange linearization with repeated nodes, 
    which is a special case of Hermite linearization, where it
    is possible to interpolate a polynomial imposing the value
    of its first derivative at the nodes. By using the 
    product dual bases it is possible to extended this construction
    to higher order derivatives. However, such a construction would have
    redundancy in the polynomial family, thus leading to linearizations
    which has infinite eigenvalues.     
    In Section~\ref{sec:hermitebasis} we present a direct 
    construction of the dual basis for the Hermite basis that does 
    not.
  \end{remark}
  
  \subsection{Explicit construction for the Newton basis}
    \label{sec:newtonbasis}
  
  Another concrete example is the construction of the 
  Newton basis linearization. We can consider, similarly to the Lagrange
  case, a set of nodes $\sigma_1, \ldots, \sigma_n$ and assume to have
  two Newton bases, one built using $\sigma_1, \ldots, \sigma_{k}$,
  and the other built using $\sigma_{k+1}, \ldots, \sigma_n$. 
  
  To construct the linearization we need to find 
  $L_{k,\phi}(\x)$ which satisfies the requirements of Theorem~\ref{thm:twobases}. 
  A possible choice is given by the following 
  \[
    L_{k,\phi}(\x) := \begin{bmatrix}
      1 & \sigma_{k} - \x \\
      & \ddots & \ddots \\
      & & 1 & \sigma_1 - \x \\
    \end{bmatrix}, \qquad 
    \pi_{k, \phi}(\x) = \begin{bmatrix}
      \prod_{j = 1}^k (\x - \sigma_j) \\
      \vdots \\
      \x - \sigma_1 \\
      1 \\
    \end{bmatrix}. 
  \]
  The matrix $L_{k,\phi}(\x)$ 
  has the right dimensions $k \times (k+1)$, 
  full row-rank for any $\x$, and is such that
  the product 
  $L_{k,\phi}(\x) \pi_{k,\phi}(\x) = 0$. Moreover, the leading
  coefficient has full row rank so we also have the minimality
  and all the hypotheses of Theorem~\ref{thm:twobases} are satisfied. 
  
  \subsection{Linearizations in the Hermite basis}
    \label{sec:hermitebasis}
    
  Recently a linearization for polynomials expressed in the Hermite
  basis has been presented by Fassbender, P\'erez and Shayanfar 
  in \cite{heike2015sparse}. 
  
  The Hermite basis can be seen as a generalization of the Lagrange
  basis where not only the values of the functions at the nodes are
  considered, but also the values of their derivatives. 
  
  Assume that we have a set of nodes $\sigma_1, \ldots, \sigma_n$, and 
  that we have interpolated a function assigning the derivative up to the
  $s$-order, for some $s \geq 1$ (the case $s = 1$ gives the Lagrange
  basis). The order $s$ can also vary depending on the node.
  We can then consider the basis given by 
  the following vector polynomial:
  \[
    \pi_{k,\phi}(\x) = \begin{bmatrix}
      \frac{\omega(\x)}{(\x - \sigma_1)^{s_1}} \\
      \vdots \\
      \frac{\omega(\x)}{(\x - \sigma_1)} \\
      \vdots \\
      \frac{\omega(\x)}{(\x - \sigma_n)^{s_n}} \\
      \vdots \\
      \frac{\omega(\x)}{(\x - \sigma_n)} \\      
    \end{bmatrix}, \qquad 
    \omega(\x) := \prod_{j = 1}^n (\x - \sigma_j)^{s_j}, 
    \qquad k = \sum_{j = 1}^n s_j. 
  \]
  A generic polynomial expressed in this basis can be written as 
  $p(\x) = p^T \pi_{k,\phi}(\x)$ where $p$ is the 
  column vector with the coefficients
  in the Hermite basis. 
  
  We want to show that it is possible to formulate a linearization for the
  Hermite basis in our framework. We already have the vector $\pi_{k,\phi}(\x)$
  so we simply need to find a linear matrix polynomial $L_{k,\phi}(\x)$
  of the correct dimension
  that has full row rank and such that $L_{k,\phi}(\x) \pi_{k,\phi}(\x) = 0$. 
  
  \begin{lemma} \label{lem:hermite-l}  
  The matrix polynomial $L_{k,\phi}(\x)$ defined as follows
  \[
    L_{k,\phi}(\x) = \begin{bmatrix}
      J_{\sigma_1}(\x) & - (\x - \sigma_2) e_{s_1} e_{s_2}^T \\
      & \ddots & \ddots \\
      & & J_{\sigma_{n-1}}(\x) & - (\x - \sigma_n) e_{s_{n-1}} e_{s_n}^T \\
      & & & \tilde J_{\sigma_n}(\x) \\
    \end{bmatrix}, 
    \]
    with
  \[
    J_{\sigma_j}(\x) := \begin{bmatrix}
      \x - \sigma_j & -1 \\
      & \ddots & \ddots \\
      & & \ddots & -1 \\
      & & & \x - \sigma_j  \\
    \end{bmatrix}, \quad 
    \tilde J_{\sigma_j}(\x) := \begin{bmatrix}
          \x - \sigma_j & -1 \\
          & \ddots & \ddots \\
          & & \x - \sigma_j & -1 \\
        \end{bmatrix},
  \]
  is a dual basis for the Hermite basis 
  $\{ \phi_i \}$ of orders $s_i$, $i = 0, \ldots, n$. 
  \end{lemma}
  
  \begin{proof}
  We can check directly that 
  $L_{k, \phi}(\x) \pi_{k, \phi}(\x) = 0$, and so 
  it only remains to verify that the row rank is maximum. We notice that for
  any $\x \neq \sigma_j$ the matrix is upper triangular with non zero
  elements on the diagonal and so the condition is obviously satisfied. 
  For $\x = \sigma_j$ then the diagonal block $J_{\sigma_j}(\x)$ 
  is singular.    
  Assume, for simplicity, that $j = 1$, and  
  consider the matrix $S$ obtained by removing the first column
  of $L_{k,\phi}(\x)$. We notice that $S$
  has the following structure: 
  \[
    S := \begin{bmatrix}
      - I \\
      0_{s_1 - 1}^T & - (\sigma_1 - \sigma_2) e_{s_2}^T \\
      & J_{\sigma_2}(\sigma_1) & - (\sigma_1 - \sigma_3) e_{s_2} e_{s_3}^T \\
      && \ddots & \ddots \\
      & & & J_{\sigma_{n-1}}(\sigma_1) & - (\sigma_1 - \sigma_n) e_{s_{n-1}} e_{s_n}^T \\
      & & & & \tilde J_{\sigma_n}(\sigma_1) \\
    \end{bmatrix}
  \]
  To prove that $S$ is invertible we consider
  the trailing submatrix $\tilde S$ obtained by removing the 
  first block row and column. We can transform $\tilde S$
  by means of block column operations so that 
  \[
    \tilde S X = \begin{bmatrix}
      \tilde u^T & \sigma_1 - \sigma_n \\
      B(\sigma_1) & - e_{k - \sigma_1 - 1} \\
    \end{bmatrix}, \qquad 
    u(\x) = \begin{bmatrix}
      - (\sigma_1 - \sigma_2) e_{s_2} \\
      \vdots \\
      - (\sigma_1 - \sigma_{n-1}) e_{s_n} \\
      0_{s_{n} - 1} \\ 
    \end{bmatrix}
  \]
  and $B(\sigma_1)$ is block diagonal with the $J_{\sigma_j}(\sigma_1)$ of size $s_j$
  on the block diagonal, except the last one which is of size $s_{n}-1$. 
  Since $u^T B(\sigma_1)^{-1} e_{k - \sigma_1 - 1} = 0$ we can write
  \begin{align*}
    \det S\tilde X &= \det B(\sigma_1) \cdot \left[ 
      (\sigma_1 - \sigma_n) + \tilde u^T B(\sigma_1)^{-1}
      e_{k-\sigma_1-1} 
      \right]= \\
      &= \prod_{j = 1}^{n-1} (\sigma_1 - \sigma_j)^{s_j}
      \cdot (\sigma_1 - \sigma_n)^{s_n - 1} (\sigma_1 - \sigma_n) = \prod_{j = 1}^n (\sigma_1 - \sigma_j)^{s_j} \neq 0. 
  \end{align*}
  This proves that $S\tilde X$ is invertible, 
  concluding the proof. 
  \end{proof}
  
  The above lemma guarantees the applicability 
  of Theorem~\ref{thm:twobases} for the case of
  Hermite polynomials (and matrix polynomials), so we 
  have an explicit way of building linearizations
  in this basis. 
  
  \subsection{Bernstein basis}
  
  A last example that is relevant in the context of computer aided design
  is the Bernstein basis, which is the building block
  of B\'ezier curves \cite{bini2005structured,farin2002history,farin2014curves,farouki2003construction}. Given an interval $[\alpha, \beta]$, we can define
  the family of Bernstein polynomials of degree $k$
  as follows: 
  \[
    \phi_i(\x) := \binom{k}{i} (\x - \alpha)^i(\beta - \x)^{k-i}, \qquad 
    i = 0, \ldots, n
  \]
  We show that also these polynomials fit in our
  construction. 
  
  \begin{lemma} \label{lem:bernstein-recurrence}
    The linear matrix polynomial $L_{k, \phi}(\x)$ defined as follows
    \[
      L_{k, \phi}(\x) := \begin{bmatrix}
        \binom{k}{k - 1} (\x - \beta) & \binom{k}{k} (\x - \alpha) \\ 
        & \ddots & \ddots \\
        && \binom{k}{0} (\x - \beta) & \binom{k}{1} (\x - \alpha) \\
      \end{bmatrix}
    \]
    is a dual minimal basis for the Bernstein polynomials of degree $k$
    defined above. 
  \end{lemma}
  
  \begin{proof}
    A direct computation shows that $L_{k, \phi}(\x) \pi_{k, \phi}(\x) = 0$. 
    Moreover, notice that for any $\x \neq \beta$ the first $k$
    columns of $L_{k, \phi}(\x)$ form a square upper triangular matrix
    with non-zero diagonal elements, and for any $\x \neq \alpha$
    the last $k$ columns are an invertible lower triangular matrix. 
    This guarantees that the row rank is maximum for any $\x \in \mathbb F$.
    Since the leading coefficient has the first $k$ columns which are
    upper triangular and invertible we also have the minimality.
  \end{proof}
    
  \section{Linearizing sums of polynomials and rational functions}
    \label{sec:sums}
    
    \subsection{Linearizing the sum of two polynomials}
    \label{sec:polynomialsum}
    
    In this section we present another example of linearization
    which deals with the following problem: assume that we are given two
    polynomials $p(\x)$ and $q(\x)$ of which we want to find the intersections, 
    that is the values of $\x$ such that $q(\x) = p(\x)$, and assume that 
    $p(\x)$ and $q(\x)$ are expressed in different bases. 
    
    Normally one would solve the problem by considering the 
    polynomial $r(\x) = p(\x) - q(\x)$ and finding its roots, for example, 
    by using a linearization. However, this requires to perform a change of basis
    on at least one of the two polynomials, and this operation is possibly
    ill-conditioned (see \cite{gautschi1983condition} for a related
    analysis). 
    
    In the case of interpolation bases, such as Newton or Lagrange, this could
    be also useful when one wants to find intersections of functions 
    that have been sampled in different data points. In this case it might
    even
    not be possible to resample the function (think of measured data). 
    
    Here we show how to linearize the problem directly. 
    
    \begin{theorem} \label{thm:polynomialsum}
    Let $p(\x)$
    and $q(\x)$ be two polynomials of the following 
    form:
    \[
      p(\x) := \sum_{j = 0}^\epsilon p_j \phi_j(\x), \qquad 
      q(\x) := \sum_{j = 0}^\eta q_j \psi_j(\x). 
    \]
    and let $L_{\epsilon, \phi}(\x)$ and $L_{\eta, \psi}(\x)$ be
    dual bases for $\{ \phi_i \}$ and 
    $\{ \psi_i \}$. Then the matrix polynomial 
      \[
        \mathcal L(\x) := \begin{bmatrix}
          p w_{\psi}^T - w_{\phi} q^T & L_{\epsilon, \phi}^T(\x) \\
          L_{\eta, \psi}(\x) & 0 \\
        \end{bmatrix}, \qquad 
        w_{\star} := \Gamma_\star^{-1}(1), \qquad 
        \star \in \{ \phi, \psi \}
      \]
      where $\Gamma_\star^{-1}(1)$ is the vector of the coefficients of the 
      constant $1$ in the basis $\star$ (see Remark~\ref{rem:basis-map}), 
      is a linearization for $r(\x) := p(\x) - q(\x)$. 
    \end{theorem}
    
    \begin{proof}
      $w_{\star} := \Gamma_\star^{-1}(1)$ means that $w_\star \pi_{k, \star}(\x) = 1$ for $\star \in \{ \phi, \psi \}$ where $k$
      is either $\epsilon$ or $\eta$ depending on the choice. By Lemma~\ref{lem:recurrence-relation} we know that $\mathcal L(\x)$ is a 
      linearization for 
      \[
        \pi_{\epsilon, \phi}^T(\x) (p w_{\psi}^T - w_{\phi} q^T) \pi_{\eta,\psi}(\x)
        = p(\x) \cdot 1 - 1 \cdot q(\x) = r(\x). 
      \]
      This concludes the proof. 
    \end{proof}
    
    \begin{remark}
      We notice that the linearization above, according to Theorem~\ref{thm:twobases}, 
      is a linearization for a polynomial of degree
       $d := \epsilon + \eta$, but
       $r(\x)$ is of degree $\max \{ \epsilon, \eta \} \leq d$. The reason for this
       is that this is actually a linearization for a polynomial of \emph{grade} $d$
       that could have some leading zero coefficients, thus having degree smaller
       than $d$. The grade is defined as the maximum degree of the 
       monomials, while the degree is the maximum of the
       \emph{non-zero} ones. 
       
       The difference between grade and degree cause 
       infinite eigenvalues to appear when we solve
       the eigenvalue problem obtained through Theorem~\ref{thm:polynomialsum}
       numerically. However, the finite eigenvalues that we get
       are still the actual roots of $r(\x)$. 
    \end{remark}
    
    The framework of Section~\ref{sec:morebases} could be used to extend the above
    result to the sum of an arbitrary number of polynomials (possibly all expressed
    in different bases). This can be obtained by combining
    the proof of Theorem~\ref{thm:polynomialsum} with the result of Theorem~\ref{thm:morebases}. 
    
    \begin{numerical-experiment} \label{numexp:chebyshev-monomial}
      In this example we test the framework on the following example. Let 
      $p_1(\x) = \sum_{i = 0}^n p_{i,1} \x^i$ and
      $p_2(\x) = \sum_{i = 0}^n p_{i,2} T_i(\x)$ 
      be two polynomials expressed
      in the monomial and Chebyshev basis of the first kind, respectively. We want
      to find the roots of their sum $q(\x) = p_1(\x) + p_2(\x)$. 
      The columns of Table~\ref{tab:chebyshev-monomial} represent,
      in the following order, 
      the result of these different
      approaches to solve the problem 
      that we tested: 
      \begin{enumerate}
        \item Converting $p_2(\x)$ to the monomial basis and using
          the Frobenius linearization to find the roots of the sum
          (by means of the command \verb|roots| in MATLAB). 
        \item Converting $p_1(\x)$ to the Chebyshev basis and using
          the colleague linearization \cite{good1961colleague,barnett1975companion}
          to find the roots of the sum of $p_1(\x)$ and $p_2(\x)$. The colleague 
          pencil has been solved using the QZ method in MATLAB. 
        \item Constructing the linearization of Theorem~\ref{thm:polynomialsum}
          and solving it with the QZ method (using \verb|eig| in MATLAB). 
        \item Constructing the linearization of Theorem~\ref{thm:polynomialsum}
          and deflating the spurious 
          infinite eigenvalues by means of the strategy
          that will be proposed in Section~\ref{sec:infinity-deflation}\footnote{The approach of Section~\ref{sec:infinity-deflation} has been moved to the end
          of this work because it is fairly general and
          does not add much information about the structure of this
          linearization. Moreover, it can be applied to other examples
          that will follow.}. 
      \end{enumerate}
      The polynomials have also been, by means of symbolical computations, 
      converted to
      the monomial basis and the roots have been computed using 
      \verb|MPSolve|\cite{bini2014solving} to 
      guarantee $16$ accurate digits. These results have been used 
      as a reference to measure the errors, which have been summarized 
      in Table~\ref{tab:chebyshev-monomial}
      and Figure~\ref{fig:chebyshev-monomial}. In
      all the cases the infinite eigenvalues have either been
      deflated a priori, or have been detected by the QZ algorithm
      and so we could deflate them a posteriori, so the numbers that we 
      report refer to the errors on the finite eigenvalues. 
      In particular, we reported the $2$-norm
    of the vectors containing the absolute errors in the computed roots.
    The coefficients of the polynomials have been generated by using the 
    \texttt{randn} function. Each experiment has been repeated $50$ times
    and only the average error is reported. 
    
    The bad results obtained in the cases where a basis conversion has
    been performed can be explained by looking at the conditioning of the
    matrix representing the change of basis between monomial and the
    Chebyshev bases. 
    
    The conditioning is 
    exponentially growing (see \cite{gautschi1983condition} 
    for a related discussion), and 
    as $n$ grows above $50$ it 
    cannot be guaranteed to compute even a single correct digit in double
    precision (see Figure~\ref{fig:changeofbasis-conditioning}, where
    the exponential growth is clearly visible), 
    and so the results start to deteriorate very quickly. 
    
    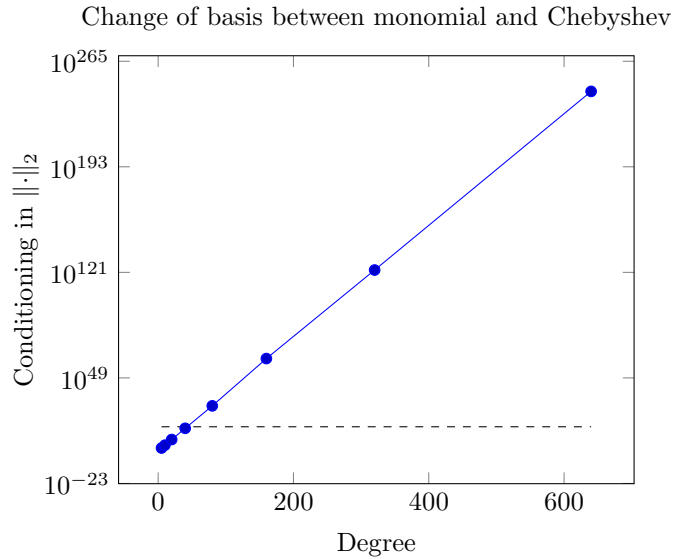
\begin{figure} \begin{center}
      \begin{tikzpicture}
      \begin{semilogyaxis}[xlabel=Degree,
        ylabel=Conditioning in $\norm{\cdot}_2$,
        title=Change of basis between monomial and Chebyshev]
      \addplot table {
         5  1.44e+01 
         10  1.29e+03 
         20  8.80e+06 
         40  4.01e+14 
         80  8.25e+29 
         160  1.48e+62 
         320  3.73e+122 
         640  2.73e+244 
      };
      \addplot[domain=5:640,dashed] {1 / (2.22e-16)};
      \end{semilogyaxis}
      \end{tikzpicture}
      \caption{Conditioning of the change of basis matrix 
        between the monomial and Chebyshev basis. The dashed line
        represents the level $\frac 1 u$, where $u$ is the machine
        precision. Beyond that point no correct digits can be guaranteed
        on the computed coefficients.}
      \label{fig:changeofbasis-conditioning}
    \end{center} \end{figure}

      \begin{table} \begin{center}
        \begin{tabular}{c|cccc}
  		  Degree & Monomial & Chebyshev & Theorem~\ref{thm:polynomialsum} & Theorem~\ref{thm:polynomialsum} + deflation
  		  \\ \hline
 5 & 7.03e-16 & 5.19e-16 & 5.44e-16 & 8.40e-16 \\ 
 10 & 1.23e-14 & 2.07e-15 & 2.00e-15 & 2.33e-15 \\ 
 20 & 1.95e-11 & 4.48e-15 & 2.49e-15 & 4.08e-15 \\ 
 40 & 1.25e-04 & 1.15e-14 & 5.59e-15 & 6.45e-15 \\ 
 80 & 1.29e+00 & 7.62e-09 & 9.76e-15 & 1.69e-14 \\ 
 160 & 4.37e+00 & 1.05e-01 & 6.90e-14 & 3.63e-14 \\ 
 320 & 9.85e+00 & 2.97e+00 & 1.07e-13 & 7.57e-14 \\ 
 640 & 1.91e+01 & 1.52e+01 & 4.40e-13 & 1.27e-13 \\ 
        \end{tabular}
              \caption{Numerical errors in the computation of the (finite)
              roots of the
              polynomial $p_1(\x) + p_2(\x)$ of the numerical experiment \ref{numexp:chebyshev-monomial}. 
              }
              \label{tab:chebyshev-monomial}    
      \end{center} \end{table}  
      
      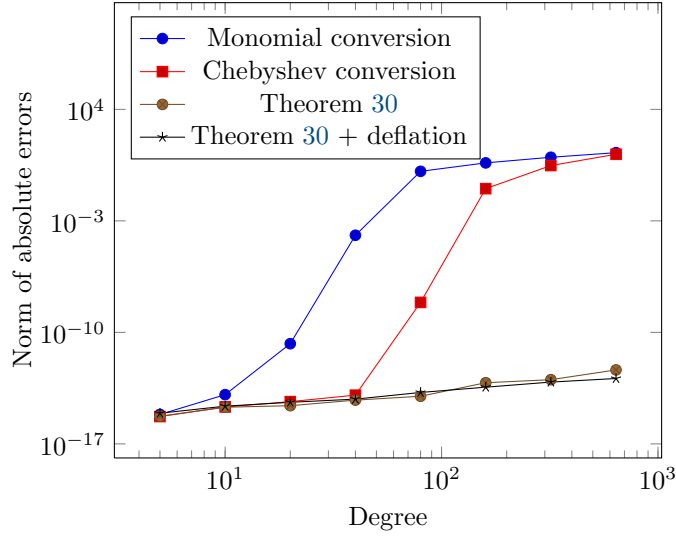
\begin{figure} \begin{center}
        \begin{tikzpicture}
        \begin{loglogaxis}[
           xlabel=Degree,
           ylabel=Norm of absolute errors,
           ymax=5e10,
           width=.7\linewidth,
           legend pos = north west
         ]
        \addplot table {
          5  7.03e-16 
          10  1.23e-14 
          20  1.95e-11 
          40  1.25e-04 
          80  1.29e+00 
          160  4.37e+00 
          320  9.85e+00 
          640  1.91e+01 
        };
        \addplot table {
         5  5.19e-16 
         10  2.07e-15 
         20  4.48e-15 
         40  1.15e-14 
         80  7.62e-09 
         160  1.05e-01 
         320  2.97e+00 
         640  1.52e+01 
        };
        \addplot table {
         5  5.44e-16 
         10  2.00e-15 
         20  2.49e-15 
         40  5.59e-15 
         80  9.76e-15 
         160  6.90e-14 
         320  1.07e-13 
         640  4.40e-13 
        }; 
        \addplot table {
         5  8.40e-16 
         10  2.33e-15 
         20  4.08e-15 
         40  6.45e-15 
         80  1.69e-14 
         160  3.63e-14 
         320  7.57e-14 
         640  1.27e-13 
        };
        \legend{Monomial conversion, Chebyshev conversion, Theorem~\ref{thm:polynomialsum},
        Theorem~\ref{thm:polynomialsum} + deflation}
        \end{loglogaxis}
        \end{tikzpicture}
        \caption{Norm of the absolute errors in the computation
        of the roots of $p_1(\x) + p_2(\x)$, where $p_1(\x)$ is a
        polynomial expressed in the monomial basis while $p_2(\x)$
        is one expresesed in the Chebyshev one.}
        \label{fig:chebyshev-monomial}        
      \end{center}
      \end{figure}
    \end{numerical-experiment}
    
    \subsection{Finding intersections of the sum of two rational functions}
    
    The results of Section~\ref{sec:polynomialsum} admit an interesting extension to finding the zeros of a sum of ratios of polynomials. 
    This has the pleasant side effect of mitigating the numerical
    issues that might be encountered when dealing with a large number
    of infinite eigenvalues.     
    Let $f(\x)$ 
    be a rational function of the form
    \[
      f(\x) := \frac{p(\x)}{q(\x)} + \frac{r(\x)}{s(\x)},
    \]
    with $p(\x), q(\x), r(\x)$, and $s(\x)$ polynomials, 
    of which we want to find the zeros. 
    We assume, in the following, that the numerators do not share
    any common factor with the denominators, and that the two ratios
    do not have common poles, otherwise the roots of the 
    common factors will also be obtained as eigenvalues of the linearization.     
    With this assumption, we have that the roots of $f(\x)$
    are the ones of $f(\x) q(\x) s(\x)$ that is 
    of the polynomial
    \[
      t(\x) := p(\x) s(\x) + r(\x) q(\x). 
    \]
    In this section we will
    linearize the polynomial $t(\x)$. However, 
    for simplicity we will sometimes inappropriately say that a linearization
    for $t(\x)$ is also a linearization for $f(\x)$, since they share the same
    zeros. 
    
    For simplicity we first consider the case in which all the polynomials
    are given in the monomial basis, and we 
    will handle the case where two different bases are used to define
    the polynomials $p(\x), q(\x), r(\x)$ and $s(\x)$ later. 
    
    \begin{theorem} \label{thm:rationalsum}
      Let $f(\x) =  \frac{p(\x)}{q(\x)} + \frac{r(\x)}{s(\x)}$ 
      a rational function obtained as the sum of two rational functions
      expressed in the monomial basis (so that $p(\x), q(\x), r(\x)$ and $s(\x)$
      are all polynomials). Assume that the numerators and the denominators
      do not share any common factor. 
      Then the matrix polynomial
      \[
        \mathcal L(\lambda) = \begin{bmatrix}
          p s^T + q r^T & L_{\epsilon}^T(\x) \\
                    L_{\eta}(\x) & 0 \\
                  \end{bmatrix}
      \]
      is a linearization for $f(\x)$, where $p,q,r$ and $s$
      are the column vectors containing the coefficients of the polynomials
      (padded with some leading zeros if the dimensions do not match)
      and $L_{k}(\x)$ is the dual basis for the monomial basis
      of degree $k$. 
    \end{theorem}
    
    \begin{proof}
      It suffices to follow the same reasoning of the proof of Theorem~\ref{thm:polynomialsum}, so that we obtain that $\mathcal L(\x)$ 
      is a linearization for 
      \[
          \pi_{\epsilon}^T(\x) (p s^T + qr^T) \pi_{\eta}(\x)
        = p(\x) s(\x) + r(\x) q(\x) = f(\x) s(\x) q(\x),
      \]
      which concludes the proof. 
    \end{proof}
    
    The result can also be extended to the case where different polynomial
    bases are involved. More precisely, we have the following corollary.
    
    \begin{corollary} \label{cor:rationalsumtwobases}
      Let $p(\x), q(\x), r(\x)$ and $s(\x)$ polynomials defined as follows:
      \[
        p(\x) = \sum_{i = 0}^{\epsilon} p_i \phi_i(\x), \quad 
        q(\x) = \sum_{i = 0}^{\epsilon} q_i \phi_i(\x), \quad         
        r(\x) = \sum_{i = 0}^{\eta} q_i \psi_i(\x), \quad     
        s(\x) = \sum_{i = 0}^{\eta} s_i \psi_i(\x)         
      \]
      for some polynomial bases $\{ \phi_i \}$ and $\{ \psi_i \}$. 
      Then the matrix polynomial
      \[
        \mathcal L(\x) = \begin{bmatrix}
          p s^T + q r^T & L_{\epsilon,\phi}^T(\x) \\
          L_{\eta,\psi}(\x) & 0 \\          
        \end{bmatrix}
      \]
      is a linearization for both 
      $f_1(\x) = \frac{p(\x)}{q(\x)} + \frac{r(\x)}{s(\x)}$ and 
      $f_2(\x) = \frac{p(\x)}{r(\x)} + \frac{q(\x)}{s(\x)}$, 
      where $p,q,r$ and $s$
      are the column vectors containing the coefficients of the polynomials
      (padded with some leading zeros if the dimensions do not match).
    \end{corollary}
    
    \begin{proof}
      By following the same proof of Theorem~\ref{thm:rationalsum} we obtain that
      $\mathcal L(\x)$ is a linearization for the polynomial
      \[
        t(\x) = \pi_{\epsilon, \phi}^T(\x) (p s^T + qr^T) \pi_{\eta, \psi}(\x)
        = p(\x) s(\x) + q(\x) r(\x)
      \]
      which has the same roots as the rational functions
      \[
       f_1(\x) = \frac{p(\x)}{q(\x)} + \frac{r(\x)}{s(\x)}, 
       \qquad 
       f_2(\x) = \frac{p(\x)}{r(\x)} + \frac{q(\x)}{s(\x)}. 
      \]
      This concludes the proof. 
    \end{proof}
    
    \begin{remark}
      The above result shows that we can handle two specific cases. 
      First, the case where each rational function is defined using 
      polynomial in a
      certain basis, and second,
      the one where both the denominators and the numerators
      share a common
      basis.
    \end{remark}
    
    An application of the above results is to find the intersection
    of two rational functions. As in the previous case, 
    $\mathcal L(\x)$ is linearization for a polynomial of grade 
    $\max\{ \deg p(\x), \deg q(\x) \} + \max\{ \deg r(\x), \deg s(\x) \} + 1$
    while the degree of the polynomial $f(\x) s(\x) q(\x)$ is 
    $\max\{ \deg p(\x) + \deg s(\x) , \deg r(\x) + \deg q(\x) \}$. 
    
    Since the first quantity is always larger than the second one, 
    the linearization introduces at least one
    infinite eigenvalue.    
    However, in many interesting cases, such as when the degree of the numerator
    and the denominator are the same in each rational function, we
    only have one spurious infinite eigenvalue, that can be deflated
    easily. 
            
    The result can however be improved and, for these cases, 
    we can build a strong linearization relying on the following. 
    
    \begin{theorem} \label{thm:strongrationalsum}
      Let $f(\x)$ be a rational function with the same hypotheses and 
      notation of Corollary~\ref{cor:rationalsumtwobases}, 
      and assume that there exist two $\epsilon \times (\epsilon - 1)$
      matrices $A$ and $B$ such that
      \[
        \pi_{\epsilon, \phi}(\x) = (\x A + B) \pi_{\epsilon-1, \phi}(\x)
      \]
      Then the matrix polynomial 
      \[
        \mathcal L(\x) = \begin{bmatrix}
          (\x A + B)^T ps^T - (\x A + B)^T q r^T & L_{\epsilon-1,\phi}^T(\x) \\
          L_{\eta,\psi}(\x) & 0 \\          
        \end{bmatrix}
      \]
      is a strong linearization for $f_1(\x)$ and $f_2(\x)$.
    \end{theorem}
    
    \begin{proof}
      By applying again Theorem~\ref{thm:twobases} we obtain that $\mathcal L(\x)$ 
      is a linearization for 
      \begin{align*}
        t(\x) &= \pi_{\epsilon-1, \phi}^T(\x) \left[ (\x A + B)^T ps^T - (\x A + B)^T q r^T \right] \pi_{\eta, \psi}(\x) \\
        &= \pi_{\epsilon, \phi}^T(\x) \left( ps^T - q r^T \right) \pi_{\eta, \psi}(\x) \\
        &= p(\x) s(\x) + q(\x) r(\x). 
      \end{align*}
      Since $t(\x)$ has degree $\epsilon + \eta$, which is the size of 
      $\mathcal L(\x)$, there are no extra infinite eigenvalues, 
      and so this is a strong linearization.     
    \end{proof}
    
    \begin{remark}
      The hypotheses of Theorem~\ref{thm:strongrationalsum} are satisfied in many cases. Some concrete examples are the following:
      \begin{enumerate}[(i)]
        \item When $\{ \phi_i \}$ is a degree-graded basis for 
          $\mathbb F_k[\x]$ then $\phi_{k+1}(\x)$ has degree $k + 1$
          and we can find $a$ so that 
          $\x^{k} = a^T \pi_{k,\phi}(\x)$. If we choose 
          $\alpha$ to be the leading coefficient of $\phi_{k+1}(\x)$
          we have
          \[
            \phi_{k+1}(\x) - \x \alpha a^T \pi_{k, \phi}(\x) = 
              b^T \pi_{k, \phi}(\x)
          \]
          for some $b \in \mathbb F^{k+1}$, since the left-hand side
          has degree $k$. This implies that 
          \[
            \pi_{k+1, \phi}(\x) = \left(
              \x \begin{bmatrix}
                & \alpha a^T & \\
                0 & \dots & 0 \\
                \vdots &  & \vdots  \\
                0 & \dots & 0 \\
              \end{bmatrix} + 
              \begin{bmatrix}
                & b^T &  \\
                1 \\
                & \ddots \\
                && 1 \\
              \end{bmatrix}
            \right) \pi_{k, \phi}(\x). 
          \]
        \item When $\{ \phi_i \}$ is an orthogonal basis then it is also
        degree-graded and so the above result applies. In this 
        case, however, 
        it is very easy to get an explicit expression for $\alpha$, 
        $a$ and $b$, since they just contain the coefficients of the recurrence
        relation that allows to obtain $\phi_{k+1}(\x)$ starting
        from the previous terms. 
        \item If $\{ \phi_i \}$ is the Lagrange basis we can still find
        suitable matrices $A$ and $B$ so that the hypothesis are 
        satisfied. Assume that $\pi_{k, \phi}(\x)$ is the 
        Lagrange basis on the interpolation nodes $\sigma_1, \ldots, 
        \sigma_k$ and that $\pi_{k+1,\phi}(\x)$ has the additional
        node $\sigma_{k+1}$. Then we have
        \[
          \pi_{k+1, \phi}(\x) = \begin{bmatrix}
            & \alpha (\x - \sigma_{k}) e_1^T & \\
            \frac{\x - \sigma_{k+1}}{\sigma_k - \sigma_{k+1}} \\
            & \ddots \\
            & & \frac{\x - \sigma_{k+1}}{\sigma_1 - \sigma_{k+1}} \\
          \end{bmatrix} \pi_{k, \phi}(\x),
        \]
        where $\alpha = \frac{1}{\sigma_{k+1} - \sigma_{k}} \prod_{j=1}^{k-1} \frac{\sigma_k - \sigma_j}{\sigma_{k+1} - \sigma_j}$. 
      \end{enumerate}
      \end{remark}
      
    \begin{remark}
      Notice that the requirement needs to hold only for one of the two
      families of polynomials. If 
      the relation holds on $\{ \psi_i \}$ instead of
      $\{ \phi_i \}$ the procedure is analogous. 
    \end{remark}
    
    As a concrete example, we report here how the (non strong)
    linearization looks 
    when considering the following rational function: 
    \[
      f(\x) = \frac{2\x^2 -1}{\x^2 + \x + 3} + 
        \frac{T_1(\x) + T_0(\x)}{T_1(\x) - T_0(\x)}. 
    \]
    We have that $p, q, r, s$ are given by 
    \[
      p = \begin{bmatrix}
        2 \\ 0 \\ -1 \\
      \end{bmatrix}, \quad 
      q = \begin{bmatrix}
              1 \\ 1 \\ 3 \\
      \end{bmatrix}, \quad 
      r = \begin{bmatrix}
        1 \\ 1 \\
      \end{bmatrix}, \quad 
      s = \begin{bmatrix}
        1 \\ -1 \\
      \end{bmatrix}. 
    \]
    We get
    \[
      \mathcal L(\x) = \begin{bmatrix}
                p s^T + q r^T & L_{\epsilon,\phi}^T(\x) \\
                L_{\eta,\psi}(\x) & 0 \\          
              \end{bmatrix} = \left[ \begin{array}{cc|cc}
        3 & -1  & 1 & 0 \\
        1 & 1 & -\x & 1 \\
        2 & 4 & 0 & -\x \\ \hline 
        1 & -\x & 0 & 0\\
      \end{array} \right]. 
    \]    
    By Theorem~\ref{thm:strongrationalsum}
    we can also obtain a strong linearization for 
    $f(\x)$. In the monomial case the $A$ and $B$ matrices
    of the hypothesis are given by 
    \[
      A = e_1^{(k+1)} (e_1^{(k)})^T, \qquad 
      B = \begin{bmatrix}
        0 & \ldots & 0 \\
        \\
        & I_k \\
        \\
      \end{bmatrix},
    \]
    where $e_i^{(k)}$ is the $i$-th column of $I_k$. A straightforward
    application of the theorem yields the linearization
    \[
      \mathcal L(\x) = \left[ \begin{array}{cc|c}
        3\x + 1 & 1 - \x & 1 \\
        2 & 4 & -\x \\ \hline
        1 & -\x & 0 \\
      \end{array} \right]
    \]
    which is a strong linearization for the rational function $f(\x)$. 
    
  In the following we report some numerical experiments that show the
  effectiveness of the approach. 
  
  \begin{numerical-experiment} \label{numexp:rational}
    Here we test the linearization for the solution of 
    the sum of rational functions. We generate four polynomials
    $p(\x), q(\x), r(\x)$ and $s(\x)$ of the same degree $n$,
    and with $p(\x), q(\x)$ being in the monomial basis and $r(\x)$ and $s(\x)$
    in the Chebyshev one. 
    
    We then find the zeros of the rational function 
    \[
      f(\x) := \frac{p(\x)}{q(\x)} + \frac{r(\x)}{s(\x)}
    \]
    by applying Theorem~\ref{thm:rationalsum} and 
    Theorem~\ref{thm:strongrationalsum} and using the 
    QZ algorithm on the obtained linearizations. 
    We compare the results with those
    obtained by symbolically computing the coefficients 
    of the polynomial $t(\x) := p(\x) s(\x) + r(\x) q(\x)$ and 
    computing its roots with $16$ guaranteed digits using 
    \texttt{MPSolve} \cite{bini2014solving}. The experiments 
    have been repeated $50$ times and an average has been taken. 
    The results are reported in 
    Table~\ref{tab:rational} and Figure~\ref{fig:rational}.
    
    \begin{table}
    \begin{center}
      \begin{tabular}{c|cc}
        Degree & Theorem~\ref{thm:rationalsum} & Theorem~\ref{thm:strongrationalsum} \\ \hline
 5 & 3.29e-16 & 2.98e-16 \\ 
 10 & 4.37e-16 & 4.01e-16 \\ 
 20 & 5.47e-16 & 5.07e-16 \\ 
 40 & 6.87e-16 & 5.75e-16 \\ 
 80 & 1.14e-15 & 7.93e-16 \\ 
 160 & 1.72e-15 & 1.40e-15 \\ 
 320 & 2.57e-15 & 2.06e-15 \\ 
 640 & 4.21e-15 & 3.53e-15 \\
      \end{tabular}
      \caption{Norm of the absolute error on the computed (finite)
      roots of the
      rational function $f(\x)$.}
      \label{tab:rational}
      \end{center}
    \end{table}
    
    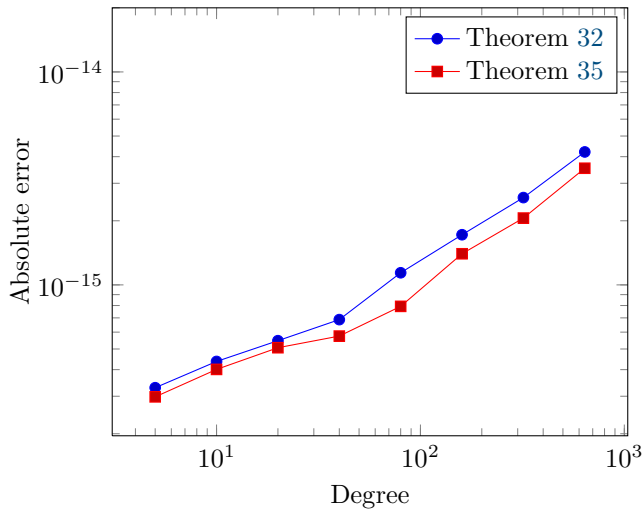
\begin{figure}
    \begin{center}
      \begin{tikzpicture}
      \begin{loglogaxis}[xlabel=Degree,ylabel=Absolute error,ymax=2e-14]
      \addplot table { 
 5  3.29e-16 
 10  4.37e-16 
 20  5.47e-16 
 40  6.87e-16 
 80  1.14e-15 
 160  1.72e-15 
 320  2.57e-15 
 640  4.21e-15 
      };
    \addplot table { 
 5  2.98e-16 
 10  4.01e-16 
 20  5.07e-16 
 40  5.75e-16 
 80  7.93e-16 
 160  1.40e-15 
 320  2.06e-15 
 640  3.53e-15  
    };
      \legend{Theorem~\ref{thm:rationalsum},
      Theorem~\ref{thm:strongrationalsum}}      
      \end{loglogaxis}
      \end{tikzpicture}
      \caption{Norm of the absolute error on the computed roots of the
            rational function $f(\x)$. Both the strong and non-strong version
            of the linearization have been tested.}
      \label{fig:rational}
          \end{center}
    \end{figure}
  \end{numerical-experiment}
    
  \section{Preserving even, odd and palindromic structures}
    \label{sec:structure-preservation}
  
  In this section we deal with the following problem: we consider the case
  where a matrix 
  polynomial has a $\star$-even, $\star$-odd or 
  $\star$-palindromic structure. These are often found in
  applications and are of particular interest since they induce
  some symmetries on the spectrum. 
  
  For this reason it is important to develop linearizations that 
  enjoy the same structure, so the symmetries in the spectrum
  will be preserved. Many authors have investigated
  this in recent years, providing different solutions
  \cite{mackey2006structured,mackey2009numerical}. Linearizations 
  for these structures have been found by exploiting the
  generality of the $\mathbb L_1$ and $\mathbb L_2$ spaces
  of linearizations introduced in \cite{mackey2006vector}.
  Our approach here leads to very similar results, but 
  is instead based on the freedom that we have in choosing the
  polynomial families $\{ \phi_i \}$ and $\{ \psi_i \}$. 
  
  Here we often use $\star$ in place of the transpose or 
  conjugate transpose operator, since the constructions
  are valid for both choices. We give the definitions
  of these structures. 
  
  \begin{definition}
    A matrix polynomial $P(\x)$ 
    is \emph{$\star$-even} if $P(\x) = P(-\x)^\star$. 
    Similarly, we say that 
    $P(\x)$ is \emph{$\star$-odd} if $P(\x) = - P(-\x)^\star$. 
  \end{definition}   
  
  \begin{definition}
    A matrix polynomial $P(\x)$ is said to be \emph{$\star$-palindromic}
    if $P(\x) = \rev P(\x)^\star$. Similarly. we say that $P(\x)$ is 
    \emph{anti $\star$-palindromic} if $P(\x) = - \rev P(\x)^\star$. 
  \end{definition}
  
  Notice that all these relations induce a certain symmetry on the coefficients
  in the monomial basis. In particular, we have the following: 
  
  \begin{lemma}
    Let $P(\x)$ a matrix polynomial. Then, 
    \begin{enumerate}[(i)]
      \item If the matrix polynomial is $\star$-palindromic 
        or anti $\star$-palindromic the eigenvalues come in
        pairs $(\x, \frac 1 {\x})$ when $\star = {}^T$
        and $(\overline\x, \frac 1 {\x})$ when $\star = {}^H$
      \item If the matrix polynomial is $\star$-even or 
      $\star$-odd the eigenvalues come in pairs $(\x, -\x)$
      when $\star = {}^T$ and $(\x, -\overline\x)$ when $\star = {}^H$.
    \end{enumerate}
  \end{lemma}
  
  \begin{proof}
    Follows immediately by the definitions above 
    noting that singularity is preserved
    by all the symmetries considered there. 
  \end{proof}
  
  \subsection{Even and odd polynomials}
  
  In this section we deal with linearizing even and 
  odd polynomials. In practice we will
  only consider the case of even polynomials since the other is 
  analogous. 
  
  \begin{theorem} \label{thm:even-linearization}
    Let $P(\x) = \sum_{i = 0}^{2k+1} P_i \x^i$ 
    be a $\star$-even matrix polynomial of 
    grade $2k + 1$. Then the even matrix polynomial
    \[
      \mathcal L(\x) = 
        \left[ \begin{array}{cccc|ccc}
              (-1)^{k} (\x P_{2k+1} + P_{2k}) &&&& I \\
              & \ddots &&& -\x I & \ddots \\
              && \ddots &&& \ddots & I  \\
              &&& \x P_1 + P_0 & && -\x I \\ \hline 
              I & \x I && \\
              & \ddots & \ddots&& \\
              && I & \x I \\ 
            \end{array} \right]
    \]
    is a linearization for $P(\x)$. 
  \end{theorem}
  
  \begin{proof}
    It is immediate to verify that the matrix polynomial is $\star$-even. In order
    to check that it is a linearization for the correct polynomial we can
    see that the top-right block and bottom-left block are of the form
    $L_{k, \phi}(\x) \otimes I_m$ and $L_{k, \psi}(\x) \otimes I_m$, respectively,
    with $L_{k, \star}(\x)$ being dual bases for
    \[
      \pi_{k, \phi}(\x) = \begin{bmatrix}
        \x^{k-1} \\
        \vdots \\
        \x \\
        1 \\
      \end{bmatrix}, \qquad 
      \pi_{k, \psi}(\x) = \begin{bmatrix}
         (-1)^{k-1} \x^{k-1} \\
         \vdots \\
        -\x \\
        1 \\
       \end{bmatrix}. 
    \]
    Applying Theorem~\ref{thm:twobases} yields that $\mathcal L(\x)$ 
    is a linearization for the matrix polynomial
    \[
      (\pi_{k, \phi}(\x) \otimes I_m)^T \diag ((-1)^j (\x P_{2j+1} + P_{2j}))_{j = 0, \ldots, k} 
      (\pi_{k, \psi}(\x) \otimes I_m) = P(\x),
    \]
    which concludes the proof. 
  \end{proof}
  
  \subsection{Palindromic polynomials}
  
  A similar procedure can be applied to obtain $\star$-palindromic
  linearizations for $\star$-palindromic polynomials. However, 
  the construction in this case is slightly more complicated. 
  We first prove the following lemma, which provides linearizations
  with $\star$-palindromic off-diagonal blocks, and then we show
  how to choose the top-left block to make the whole matrix polynomial
  $\star$-palindromic. 
  
  \begin{theorem} \label{thm:palindromic-linearization}
    Let $\{ \phi_i \}$ and $\{ \psi_i \}$ be the polynomial bases
    defined by 
    \[
      \phi_i = \x^{k - i}, \qquad 
      \psi_i = \x^i, \qquad 
      i = 0, \ldots, k.
    \]
    Then two dual bases $L_{k, \phi}(\x)$ and $L_{k, \psi}(\x)$ 
    for $\{ \phi_i \}$ and $\{ \psi_i \}$, respectively, 
    are given by 
    \[
      L_{k, \phi}(\x) = \begin{bmatrix}
        1 & -\x \\ 
        & \ddots & \ddots \\
        && 1 & -\x \\
      \end{bmatrix}, \qquad L_{k, \psi}(\x) = \begin{bmatrix}
              \x & -1 \\ 
              & \ddots & \ddots \\
              && \x & -1 \\
            \end{bmatrix}
    \] and the $\star$-palindromic matrix polynomial
    \[
      \mathcal L(\x) = \begin{bmatrix}
        \x M + M^\star & L_{k, \phi}(\x)^\star \otimes I_m \\
        L_{k, \psi}(\x) \otimes I_m & 0 \\
      \end{bmatrix}, \quad 
      M = \begin{bmatrix}
              M_{1,1} & \dots & M_{1,k} \\
              \vdots && \vdots \\
              M_{k,1} & \dots & M_{k,k} \\ 
            \end{bmatrix},
    \]
    where $M_{ij} \in \mathbb C^{m \times m}$, 
    is a linearization for the degree $2k-1$ 
    matrix polynomial defined by
    \[ P(\x) = \sum_{i, j = 1}^k M_{j,i}^\star \x^{k + j - i - 1}
    + \sum_{i, j = 1}^k M_{i,j} \x^{k + j - i}. \]
  \end{theorem}
  
  \begin{proof}
    It is immediate to verify that the given matrices $L_{k, \phi}(\x)$
    and $L_{k, \psi}(\x)$ are indeed dual bases.     
    By applying Theorem~\ref{thm:twobases} we get
    $P(\x)$ as 
    \[
      P(\x) = \begin{bmatrix} \x^{k-1} I_m & \cdots & I_m \end{bmatrix} (\x M + M^\star) 
      \begin{bmatrix}
        I_m \\ 
        \vdots \\
        \x^{k-1} I_m
      \end{bmatrix}.
    \]
  \end{proof}
  
  The result above can be used to construct $\star$-palindromic linearizations
  for $\star$-palindromic matrix polynomials. Let 
  $P(\x) = \sum_{j = 0}^n P_j \x^j$ be such a polynomial, and assume we
  want to describe a procedure to choose the block coefficients
  of $M$ in Theorem~\ref{thm:palindromic-linearization} in order
  to make $\mathcal L(\x)$ a linearization for $P(\x)$.   
  
  \begin{definition}[Block Diagonal Sum]
    Let $X = \bsd_k(M, d)$ be the matrix defined as the sum of the 
    matrices along the $d$-th block diagonal of the matrix $M$ 
    partitioned in blocks of size $k$. We refer to $X$ as the 
    \emph{$d$-th block diagonal sum of $M$}. Whenever the block
    matrix $M$ does not have a $d$-th diagonal (since it is too small), 
    we define $\bsd_k(M, d)$ to be the zero matrix. 
  \end{definition}
  
  \begin{remark} \label{rem:palindromic-coefficients}
    The linear matrix polynomial $\mathcal L(\x)$ defined in Theorem~\ref{thm:palindromic-linearization} is a linearization
    for a matrix polynomial $P(\x)$ of degree $2k - 1$ if and only if
    the relation 
    \[
      P_j = 
        \bsd_m(M, j - k) + \bsd_m(M, k - j - 1)^\star 
    \]
    holds for any $j = 0, \ldots, 2k - 1$. 
  \end{remark}
  
  Notice that Remark~\ref{rem:palindromic-coefficients} can also be used
  to build the linearization starting from its coefficients. In fact, 
  the relation for $j \in \{ 0, 2k - 1 \}$ simplifies to:
  \[
    P_0 = M_{1,k}^\star, \qquad 
    P_{2k-1} = M_{1, k}. 
  \]
  Having determined the term in position $(1,k)$, one can
  then proceed to fill in the others by imposing the condition
  of Remark~\ref{rem:palindromic-coefficients}. 
  
  Here we provide a concrete example of such a construction. However, 
  we stress that is not the only possible choice. 
  
  \begin{theorem} \label{thm:concrete-palindromic-linearization}
    Let $P(\x) = \sum_{i = 0}^n P_i \x^i$ a degree $2k-1$ and $\star$-palindromic matrix polynomial. 
    Then the matrix polynomial of Theorem~\ref{thm:palindromic-linearization}
    with
    \[
      M = \begin{bmatrix}
        0_m & \cdots & 0_m & P_0^\star \\
        \vdots && \vdots & \vdots \\
        0_m & \cdots & 0_m & P_{k-1}^\star \\
      \end{bmatrix}
    \]
    is a $\star$-palindromic linearization for $P(\x)$. 
  \end{theorem}
  
  \begin{proof}
   Notice that, in the formula above, we have that the only 
   non zero diagonal elements of $M$ are on the last column and 
   $M_{i,k} = P_{i-1}^\star = P_{2k-i}$. We can check that the equality
   of Remark~\ref{rem:palindromic-coefficients} holds. 
   If $0 \leq j \leq k - 1$ we have $\bsd_m(M, j - k) = 0$ 
   and $\bsd_m(M, k - j - 1)^\star = M_{i, k}^\star$ where $i$ 
   is such that $k - i = k - j - 1$ (being on the $(k-j-1)$-th
   diagonal). This implies that $i = j + 1$ and so $M_{i,k}^\star = P_{j}^\star$, 
   as desired. On the other hand, if $k \leq j \leq 2k-1$ we similarly have
   $\bsd_m(M, k - j - 1) = 0$ and $\bsd_m(M, j - k) = M_{i,k}$ with 
   $k-i = j - k$ so that $i = 2k - j$. This again implies that $M_{i,k} = P_{2k-i} = P_{2k - (2k - j)} = P_j$. This concludes the proof. 
  \end{proof}
  
  \section{Deflation of infinite eigenvalues}
    \label{sec:infinity-deflation}
    
  We have observed that in the polynomial sum case of Section~\ref{sec:polynomialsum} the linearization built according
  to Theorem~\ref{thm:polynomialsum} is generally not strong and might have
  many infinite eigenvalues. 
  
  In this section we show what the structure of the infinite eigenvalues
  is and a possible strategy to deflate them
  based on 
  a simplified approach inspired by 
  \cite{beelen1988improved,van1979computation}. 
  In our case we have the advantage of knowing
  exactly which eigenvalue we want to deflate and we can completely
  characterize the structure of the block in the Kronecker
  canonical form corresponding to the infinite eigenvalue. 
  
  \begin{lemma} \label{lem:infinite-eigenstructure}
    Let $\mathcal L(\x)$ be the lineraization obtained from Theorem~\ref{thm:polynomialsum} for the sum of two arbitrary polynomials. 
    Then there exist two unitary bases $Q$ and $Z$ such that 
    \[
      Q^H \mathcal L(\x) Z = \begin{bmatrix}
        I - \x J & \x A_1 - A_0 \\
        0          & \x B_1 - B_0 \\
      \end{bmatrix}, \qquad 
      J = \begin{bmatrix}
        0 & 1 \\
          & \ddots & \ddots \\
          &        & \ddots & 1 \\
          &        &        & 0 \\
      \end{bmatrix}. 
    \]
  \end{lemma}
  
  \begin{proof}
    Such a decomposition can be obtained by following the deflation 
    procedure for the infinite eigenvalue described in 
    \cite{beelen1988improved} and \cite{van1979computation}. 
    We only need to prove that the pencil
    obtained in the top-left entry of the transformed matrix is 
    exactly $I - \x J$. 
    Let $A, B$ be matrices such that $\mathcal L(\x) = A - \x B$. 
    We note that $B$ 
    has nullity equal to $1$ in our construction. 
    Recall that the columns of 
    $Q$ and $Z$ are
    orthogonal bases of the sequence of spaces defined by 
    \[
      \mathcal Z_i = \begin{cases}
        \{ 0 \} & \text{if } i = 0 \\
        B^{-1} \mathcal Q_{i-1} & \text{otherwise} \\ 
      \end{cases}, \qquad 
      \mathcal Q_i = A \mathcal Z_i, 
    \]
    where $B^{-1}$ is the pre-image of $B$. The fact that $B$ has nullity
    $1$ implies that the dimension of $\mathcal Z_i$ can increase at most 
    of $1$ at each step. 
    This means that there exist a unique diagonal block in the 
    Kronecker canonical form corresponding
    to the infinite eigenvalue, whose size is exactly equal to the algebraic
    multiplicity of it. 
  \end{proof}
  
  We can use the algorithm described in \cite{beelen1988improved}
  to compute the matrices $Q$ and $Z$ and then solve the
  pencil $\x B_1 - B_0$ instead of 
  $\mathcal L(\x)$. Experiments using this strategy were reported in 
  Section~\ref{sec:polynomialsum}. 
  
  For a more in-depth discussion of the above approach to deflation
  see the work of Berger and Reis \cite{berger2013controllability} 
  which is based on the analysis originally carried out by Wong \cite{wong1974eigenvalue}.

  \section{Conclusions} \label{sec:conclusions}
  
  We have
  provided an extension of the main theorem of \cite{ldvdp}
  to construct linearizations. This new result 
  makes it easier to prove
  that many matrix polynomials are linearizations for, among others, 
  sum of polynomials, rational functions, and
  allows to realize structure preserving pencils. 
  
  We think that the flexibility offered by the adjustment of the
  dual basis in the pencil $\mathcal L(\x)$ allows even for further
  improvement and for the coverage of more structures. We think that in many cases
  this construction can be used as an alternative to other approaches
  to find structured linearizations, such as looking in the spaces
  $\mathbb L_1$ and $\mathbb L_2$ from \cite{mackey2006vector}. 
  
  \section*{Acknowledgements}
    
  This paper presents research results of the Belgian Network DYSCO
  (Dynamical Systems, Control, and Optimization), funded by the
  Interuniversity Attraction Poles Programme initiated by the Belgian
  Science Policy Office.
  
    We wish to thank Piers Lawrence, who helped to understand the 
    construction of dual bases for the Lagrange case, and 
    Thomas Mach for inspiring discussions.

  \bibliographystyle{plain}
  \bibliography{biblio}

\end{document}